\documentclass[10pt]{amsart}
\usepackage{amscd, amsfonts, amsthm, amsgen, amsmath, amssymb ,verbatim, enumerate, textcomp}
\usepackage{hyperref} 
\usepackage[cmtip, all]{xy}
\usepackage{graphicx}
\usepackage[margin=1in]{geometry}

\usepackage{amscd}
\usepackage{xypic}
\usepackage{amsmath, amssymb}

\newtheorem{theorem}{Theorem}[section]

\newtheorem{corollary}[theorem]{Corollary}
\newtheorem{proposition}[theorem]{Proposition}
\theoremstyle{definition}
\newtheorem{definition}[theorem]{Definition}
\theoremstyle{notation}

\theoremstyle{remark}
\newtheorem{remark}[theorem]{Remark}

\numberwithin{equation}{section}
\usepackage{amscd}
\usepackage{xypic}
\usepackage{amsmath, amssymb}

\numberwithin{equation}{subsection}
\newcommand{\be}%
  {\protect\setcounter{equation}{\value{subsubsection}}}
  \newcommand{\ee}%
   {\protect\setcounter{subsubsection}{\value{equation}}}

  {\protect\setcounter{subsubsection}{\value{equation}}}

\def \rmA{\rm A}
\def \BG{\rm BG}
\def \EG1{\rm EG}
\def \BN(T){\rm BN_{G}(T)}
\def \rmB{\rm B}

\def \CH{\rm {CH}}

\def \Cl{\mathbb C}

\def \colimm{\underset {m \rightarrow \infty}  {\hbox {lim}}}
\def \colimn{\underset {n \rightarrow \infty}  {\hbox {lim}}}

\def \colimK.{\underset {\underset K^.  \rightarrow}  {\hbox {lim}}}

\def \colimU.{\underset {\underset U_.  \rightarrow}  {\hbox {lim}}}

\def \compl{\, \, {\widehat {}}}

\def \rmD{\rm D}
\def \rmd{\rm d}

\def \DS1X{\rm {DS^1X}}
\def \rmD{\rm D}

\def \cE{\mathcal E}

\def \EG1{E{(G \times {\mathbb C}^*)}{\underset {G\times {\mathbb C}^*} 
\times}}

\def \EZ(s)1{E{(Z(s) \times {\mathbb C}^*)}{\underset {(Z(s)\times {\mathbb
C}^*)}  \times}}

\newcommand{\eps}{ \, {\boldsymbol\varepsilon} \,}

\def \EM(u){EM(u){\underset {M(u)}  \times}}
\def \EM(us){EM(u,s){\underset {M(u, s)}  \times}}

\def \EG{\rm EG}
\def \rmE{\rm E}

\def \rmf{\rm f}
\def \rmF{\rm F}

\def \group{\rm G}

\def \rmG{\rm G}
\def \GL{\rm {GL}}

\def\holimD{\mathop{\textrm{holim}}\limits_{\Delta }}
\def\holim{\mathop{\textrm{holim}}\limits}

\def \holimn {\underset {\infty \leftarrow n}  {\hbox {holim}}}
\def \holimm {\underset {\infty \leftarrow m}  {\hbox {holim}}}

\def \H{\mathbb H}
\def \rmH{\rm H}

\def \Hom{\underline {Hom}}
\def \Hom{{\mathcal H}om}

\def \rmh{\rm h}

\def \invlimn{\underset {\infty \leftarrow n}  {\hbox {lim}}}
\def \invlim1{\underset {\infty \leftarrow q}  {\hbox {lim}}^1}

\def \oJ{\rm J}

\def \rmK{\rm K}
\def \k{\it k}
\def \L3{\Lambda \times \Lambda \times \Lambda}
\def \L2{\Lambda \times \Lambda}

\def \lim{\underset \leftarrow  {\hbox {lim}}}

\def \invlimn{\underset {\infty \leftarrow n}  {\hbox {lim}}}
\def \limn{\underset {\infty \rightarrow n}  {\hbox {lim}}}

\def \longright2arrow{{\overset \longrightarrow  {\overset {} 
\longrightarrow}}}

\def \L{L\times \Cl ^*}

\def \rmL{\rm L}

\def \rmM{\rm M}
\def \M{\mathcal M}
\def \Map{\underline {Map}}
\def \Map{{\mathcal M}ap}

\def \N(T){\rm {N_{G}(T)}}
\def \rmN{\rm N}
\def \NT{\rm N_{G}(T)}
\def \Nis{\rm {Nis}}
\def \rmp{\rm p}

\def \rmP{\rm P}

\def \Spt{\rm {Spt}}

\def \rmQ{\rm Q}
\def \rmq{\rm q}

\def \ra{\rightarrow}

\def \RG^{R(G)^{\hat {}}\ }

\def \res{respectively}
\def \rmS{\rm S}

\def \RHom{{{\mathcal R}{\mathcal H}om}}
\def \rmRHom{\rm RHom}

\def \rmR{\rm R}
\def \Sm{\rm {Sm}}

\def \Speck{{\rm {Spec}}\, {\it k}}

\def\Spt{\rm {\bf Spt}}

\def \Sph{\rm {Sph}}

\def \SH{{\mathcal S}{\mathcal H}}
\def\Spt{\rm {\bf Spt}}
\def\Spc{\rm {\bf Spc}}

\def \mbS{\mathbb S}
\def \rmS{\rm S}

\def \topGcoh*{^{top, *} _{G}}
\def \topGho*{ _{top,*} ^{G}}

\def \T{{\mathbf T}}

\def \tr{\it {tr}}
\def \rmT{\rm T}

\def \rmU{\rm U}
\def \U{\mathcal U}
\def \rmV{\rm V}

\def \rmW{\rm W}

\def \rmX{\rm X}

\def \X{\mathcal X}

\def \Y{\mathcal Y}
\def \rmY{\rm Y}

\def \Z(s){Z(s) \times {\mathbb C}^*}
\def \Z{\mathcal Z}
\def \bZ{\mathbb Z}

\def \rmZ{\rm Z}

\def\Zl{{\mathbb Z}/\ell}
\begin{document}

\title{The Motivic and \'Etale Becker-Gottlieb transfer and splittings}
\author{Gunnar Carlsson}
%  Address of record for the research reported here
\address{Department of Mathematics, Stanford University, Building 380, Stanford,
California 94305}
\email{gunnar@math.stanford.edu}
\thanks{  }  
\author{Roy Joshua}
\address{Department of Mathematics, Ohio State University, Columbus, Ohio,
43210, USA}
\email{joshua@math.ohio-state.edu}
\thanks{}
%\author{Pablo Pelaez}
%\address{Instituto de Matem\'aticas, Ciudad Universitaria, UNAM, DF 04510, M\'exico.}
%\email{pablo.pelaez@im.unam.mx}

%   Current address
%\curraddr{Max Planck Institut f\"ur Mathematik,
 %         P.O. Box 7280, Bonn 53072
  %        GERMANY}

    %\thanks will become a 1st page footnote.
\thanks{\indent  Both authors were supported by  grants
from the NSF at various stages on work on this paper. The second author was also supported by a grant from the Simons Foundation.
The second author would also like to thank the Isaac Newton Institute for Mathematical Sciences, Cambridge, for support and 
hospitality during the programme {\it K-Theory, Algebraic Cycles and Motivic Homotopy Theory} where part of the work on this paper was carried 
out supported by EPSRC grant no EP/R014604/1.}
\begin{abstract} In this paper, we establish the key properties of the motivic and \'etale Becker-Gottlieb transfer  
including compatibility with \'etale and Betti realization and show how to obtain various splittings using the transfer.

\end{abstract}
\maketitle

\centerline{\bf Table of contents}
\vskip .2cm 
1. Introduction
\vskip .2cm
2. Basic properties of the transfer
\vskip .2cm
3. Review of \'etale and Betti realization
\vskip .2cm
4. Compatibility of the transfer with realizations
\vskip .2cm
5. Transfer and splittings in the motivic stable homotopy category
\vskip .2cm
6. Further details on \'etale realization
\input xypic
%\vfill \eject
\vskip .5cm
%%%%%%%%%%%%%%%%%%%%%%%%%%%%%%%%%%%%%%%%%%%%%%%%%%%%%%%%%%%%%%%%%%%%%
%%%%%%%%%%%%%%%%%%%%%%%%%%%%%%%%%%%%%%%%%%%%%%%%%%%%%%%%%%%%%%%%%%%%%
%% INTRODUCTION
%%%%%%%%%%%%%%%%%%%%%%%%%%%%%%%%%%%%%%%%%%%%%%%%%%%%%%%%%%%%%%%%%%%%%
\section{\bf Introduction}
This paper is a continuation of the paper \cite{CJ23-T1}, where the authors defined the Becker-Gottlieb transfer in the 
motivic and \'etale framework for the Borel construction. The goal of the present paper is to establish the basic properties of
 this transfer and then establish various splittings making use of it.  Therefore, we will adopt the basic framework and terminology of
 \cite[section 1]{CJ23-T1}. We will briefly recall these here.
 \vskip .2cm 
 {\it Basic assumptions on the base field}. 
\begin{enumerate}
\label{basic.assumpts.field}
 \item A {\it standing assumption throughout} is that the base field $k$ is a {perfect field of arbitrary characteristic.} \index{perfect field}
 \item When considering actions by linear algebraic groups $\rmG$ that are {\it not special}, we {\it will also assume the base field is  infinite} to prevent certain
unpleasant situations. 
\item On considering \'etale realizations of the transfer, it is important to assume
 that the base field \index{finite cohomological dimension}
 \be \begin{multline}
   \label{etale.finiteness.hyps}
    \begin{split}
 k \mbox{ has finite } \ell-\mbox{cohomological dimension,  for } \ell \ne char(k) \mbox{ and satisfies the finiteness conditions  that}\\
 	 \rmH^n_{et}(Spec \, {\it k}, {\mathbb Z}/\ell^{\nu}) \mbox{ is finitely generated in each degree } n \mbox{ and vanishes for all } n>> {\rm 0}, \mbox { all } \nu > 0. \\
 	 \mbox{(Such an assumption is not needed on dealing with the motivic transfer alone.)} 
   \end{split}
 \end{multline} \ee
 \vskip .1cm \noindent
One should be able to see that such an assumption is necessary to get any theory of Spanier-Whitehead duality on the \'etale site of $\Speck$.
\end{enumerate}
\vskip .2cm
{\it Basic assumptions on the linear algebraic groups} considered: 
\begin{enumerate}
\label{basic.assmpts.group}
 \item 
we allow any linear algebraic group over $k$, {\it irrespective} of whether 
 it is {\it connected or not} and 
\item
 we are { not assuming it is special in the sense of Grothendieck (see \cite{Ch}). This means, in particular,  we allow groups such
 as all orthogonal groups and finite groups, which are all known to be non-special.}
\end{enumerate}
 
\section{\bf Basic properties of the transfer}
We may start with a $\rmG$-torsor $\rmE \ra \rmB$, with both $\rmE$ and $\rmB$ smooth quasi-projective schemes over $\rmS$. 
We will further assume that $\rmB$ is {\it always connected}. 
\vskip .1cm
{\it The Borel construction}. 
Given a simplicial presheaf $\rmX$ with an action by $\rmG$, one forms the quotient $\rmE\times _{\rmG} \rmX$
of the product $\rmE \times \rmX$ by the diagonal action: here $\rmG$ acts on the right on $\rmE$ through the involution of $\rmG$ given by $g \mapsto g^{-1}$ and
on the left on $\rmX$ in the usual manner. The construction of such a quotient is the Borel construction and
it needs to be carried out carefully so that if $\rmX$ is a smooth scheme, one obtains the correct object. This construction is discussed
in detail in \cite[section 8.3]{CJ23-T1}.
\vskip .1cm
Let $\rmX$ and $\rmY$ denote two simplicial presheaves provided with $\rmG$-actions. We will consider the following three {\it basic contexts} for the transfer: \index{torsor} \index{The three basic contexts}
 \vskip .1cm
 (a) $\rmp: \rmE \ra \rmB$ is a $\rmG$-torsor for the action of a linear algebraic group $\rmG$ with both $\rmE$ and $\rmB$ smooth quasi-projective schemes over $k$, with $\rmB$ {\it connected} and 
 \[\pi_{\rmY}: \rmE \times_{\rmG}  (\rmY \times \rmX) \ra \rmE \times_{\rmG}  \rmY\]
 the induced map, where
 $\rmG$ acts diagonally on $\rmY \times \rmX$. One may observe that, on taking $\rmY= Spec \, \k$ with the trivial action of $\rmG$, the map $\pi_{\rmY}$ becomes $\pi:\rmE\times_{\rmG} \rmX \ra \rmB$ (the induced projection), which is an
 important special case.
\vskip .1cm 
(b) A basic example of such a $\rmG$-torsor is ${\EG}^{\it gm,m} \ra {\BG}^{\it gm,m}$, where ${\BG}^{\it gm,m}$ denotes the $m$-th degree approximation to the geometric classifying space of the linear algebraic group $\rmG$ as in \cite{MV} (see also \cite{Tot}),
$\rmp: {\EG}^{\it gm, m} \ra {\BG}^{\it gm, m}$ is the corresponding universal $\rmG$-torsor and 
\[\pi_{\rmY}: {\EG}^{\it gm, m}\times_{\rmG} (\rmY \times \rmX) \ra {\EG}^{\it gm, m}\times_{\rmG} \rmY\]
is the induced map.
\vskip .1cm
(c) If $ \rmp_m$ ($\pi_{\rmY, m}$) denotes the map denoted $\rmp$ ($\pi_{\rmY}$) in (b), here we let $\rmp = \colimm \rmp_m$ and let 
\[\pi_{\rmY} = \colimm \pi_{\rmY,m}: {\EG}^{\it gm}\times_{\rmG} (\rmY \times \rmX) =\colimm {\EG}^{\it gm, m}\times_{\rmG}(\rmY \times \rmX) \ra \colimm {\EG}^{\it gm, m}\times_{\rmG} \rmY = {\EG}^{\it gm}\times_{\rmG} \rmY.\]
Strictly speaking, the above definitions apply only to the case where $\rmG$ is { special}. %in the sense of Grothendieck (see \cite{Ch}) 
 When $\rmG$ is { not special}, the above objects will in fact need to be replaced by the derived push-forward of the above objects viewed as sheaves on the
big \'etale site of $k$ to the corresponding big Nisnevich site of $k$, as discussed in \cite[(8.3.6)]{CJ23-T1}.  {\it However, we will denote
these new objects also by the same notation throughout, except when it is necessary to distinguish between them}. Recall that, for $\rmG$ { not special}, we {\it will assume the base field is also infinite} to prevent certain
unpleasant situations. 
\vskip .2cm
Throughout the following discussion, $\cE^{\rmG}$ will denote any one of the $\rmG$-equivariant ring spectra considered in \cite[(4.0.24)]{CJ23-T1}, with $\cE$ denoting the 
 corresponding non-equivariant spectrum: see \cite[Definition 4.13, (4.0.29)]{CJ23-T1}.
\begin{definition} 
\label{weak.ring.modules}
(Weak ring and module spectra over commutative ring spectra) 
Let $\rmA$ denote a spectrum in $\Spt(\k_{\rm mot}, \cE)$ ($\Spt(\k_{et}, \epsilon^*(\cE))$). Then we call $\rmA$ a {\it weak ring spectrum}
if there is given a pairing 
\[\mu: \rmA \wedge \rmA \ra \rmA\]
in $\Spt(\k_{\rm mot}, \cE)$ ($\Spt(\k_{et}, \epsilon^*(\cE))$) that is
homotopy associative. A spectrum $\rmM \in \Spt(\k_{\rm mot}, \cE)$ ($\Spt(\k_{et}, \epsilon^*(\cE))$) is then called a {\it a weak module spectrum} over $\rmA$
if it comes equipped with a pairing $\rmA \wedge \rmM \ra \rmM$ that is homotopy associative in the sense that the obvious square involving $\mu$
and the last pairing homotopy commutes.
 \end{definition} \index{weak ring spectrum} \index{weak module spectrum}
\vskip .2cm
Then we recall from \cite[Definition 8.8, (8.4.18), (8.4.19) and (8.4.20)]{CJ23-T1} the transfer map is defined as follows. 
Let $\rmf: \rmX \ra \rmX$ denote
a ${\group}$-equivariant map and let $\pi_{\rmY}: \rmE\times _{\rm G}(\rmY \times \rmX) \ra \rmE\times _{\rm G}\rmY$ denote any one of the maps considered in {\rm (a)} through {\rm (c)} above. Let $\rmf_{\rmY} = id_{\rmY} \times \rmf:\rmY \times \rmX \ra \rmY \times \rmX$ denote
 the induced map. 
\vskip .2cm
Then in case {\rm (a)} and when the group $\rmG$ is special, we obtain a map (called {\it the transfer })
\[\tr(\rmf_{\rmY}):{\Sigma^{\infty}_{\T}}(\rmE\times _{\rm G}\rmY)_+ \ra {\Sigma^{\infty}_{\T}} (\rmE\times _G({\rmY}\times \rmX))_+\quad  (\tr(\rmf_{\rmY}): \cE \wedge (\rmE\times _{\rm G}\rmY)_+ \ra \cE \wedge (\rmE\times _G({\rmY}\times \rmX))_+ )\]
in ${\SH}(k)$ ($\SH(k, {\cE})$, \res) if ${\Sigma^{\infty}_{\T}}\rmX_+$ is dualizable in ${\SH}(k)$ 
(if $ \cE \wedge \rmX_+$ is dualizable in $\SH(k, {\cE})$, \res). In case {\rm (a)} and when $\rmG$ is not special, the corresponding transfer is the map
\[\tr(\rmf_{\rmY}):  \rmR\epsilon_*(\epsilon^*\mbS_{\k}) \wedge  \rmR\epsilon_*({\widetilde {\rmE}}\times_{\rmG}^{et} \rmY)_+  \ra \rmR\epsilon_*(\epsilon^*{\mbS_{\k}}) \wedge \rmR\epsilon_*({\widetilde {\rmE}}\times_{\rmG}^{et}(\rmY \times  \rmX))_+ \]
\[(\tr(\rmf_{\rmY}):  \rmR\epsilon_*(\epsilon^*{\cE}) \wedge \rmR\epsilon_*({\widetilde {{\rmE}}}\times_{\rmG}^{et} \rmY)_+ \ra  \rmR\epsilon_*(\epsilon^*{\cE}) \wedge \rmR\epsilon_*({\widetilde {{\rmE}}}\times_{\rmG}^{et} (\rmY \times \rmX))_+, \res).\]
\begin{definition}
 \label{coh.not}
 Let $\rmA$ denote a weak ring spectrum in $\Spt(\k_{\rm mot}, \cE)$, and let $\rmM$ denote a weak module spectrum over $\rmA$. When $\rmG$ is special, and $\rmZ$ is a simplicial presheaf with $\rmG$-action,
 we will let $\rmh^{*, \bullet}(\rmE\times _{\rm G}\rmZ, \rmM)$ denote the generalized motivic cohomology of $\rmE \times_{\rmG}\rmZ$ with 
 respect to the motivic spectrum $\rmM$. 
 \vskip .1cm
When $\rmG$ is non-special and $\rmZ$ is a simplicial presheaf with $\rmG$-action, we will let 
 \[\rmh^{i, j}(\rmE\times _{\rm G}\rmZ, \rmM) = [\rmR\epsilon_*(\epsilon^*(\cE)) \wedge \rmR\epsilon_*(\rmE\times _{\rm G}\rmZ)_+, \rmM(j)[i]], \]
and where $[\quad, \quad ]$ denotes hom in $\SH(\k, \cE)$.
\end{definition}

\begin{theorem}
\label{thm.1}
The transfer has the following properties.
\begin{enumerate}[\rm(i)]
\item If $\tr(\rmf_{\rmY})^m:{\Sigma^{\infty}_{\T}} (\EG^{\it gm, m}\times_{\rmG}\rmY) _+ \ra {\Sigma^{\infty}_{\T}} (\EG^{\it gm,m}\times_{\rmG}(\rmY \times \rmX))_+\quad  (\tr(\rmf_{\rmY})^m: \cE \wedge (\EG^{\it gm, m}\times_{\rmG}\rmY) _+ \ra \cE \wedge (\EG^{\it gm,m}\times_{\rmG}(\rmY \times \rmX))_+ )$
denotes the corresponding transfer maps in case {\rm (b)}, the maps $\{\tr(\rmf_{\rmY})^m|m\}$ are compatible as $m$ varies. The corresponding induced map $\colimm \, \tr(\rmf_{\rmY})^m$ will be denoted $\tr(\rmf_{\rmY})$.
\vskip .1cm 
For items {\rm (ii)} through {\rm (v)} we will assume that $\rmA$ is a weak ring spectrum in $\Spt(\k_{\rm mot}, \cE)$ and 
that $\rmM$ is a weak module spectrum over $\rmA$. Moreover, for items {\rm (ii)} through {\rm (iv)} we may assume 
any one of the above contexts {\rm (a)} through {\rm (c)}.
\vskip .1cm
\item If $\rmh^{*, \bullet}(\quad, \rmA)$ $(\rmh^{*, \bullet}( \quad, \rmM)$) denotes the generalized motivic cohomology theory defined as in Definition ~\ref{coh.not} with respect to the weak ring spectrum $\rmA$  (the weak module spectrum $\rmM$ over $\rmA$, \res) 
 then,
 \[\tr(\rmf_{\rmY})^*(\pi_{\rmY}^*(\alpha). \beta) = \alpha .\tr(\rmf_{\rmY})^*(\beta), \quad \alpha \in \rmh^{*, \bullet}(\rmE\times _{\rmG}\rmY, \rmM), \beta \in \rmh^{*, \bullet}(\rmE\times _{\rmG}(\rmY \times \rmX), \rmA).\]
Here $\tr(\rmf_{\rmY})^*$ ($\pi_{\rmY}^*$) denotes the map induced on generalized cohomology by the map $\tr(\rmf_{\rmY})$ ($\pi_{\rmY}$, \res). Both $\tr(\rmf_{\rmY})^*$ and $\pi_{\rmY}^*$ preserve the degree as well as the weight.
\item
The composition $\tr(\rmf_{\rmY})^* \circ \pi_{\rmY}^*:\rmh^{*, \bullet}(\rmE\times _{\rmG}\rmY, \rmA) \ra \rmh^{*, \bullet}(\rmE\times _{\rmG}\rmY, \rmA)$ is an {\it isomorphism} if and only if
$\tr(\rmf_{\rmY})^*(1) = \tr(\rmf_{\rmY})^*(\pi_{\rmY}^*(1))$ is a unit in $\rmh^{0, 0}(\rmE\times _{\rmG}\rmY, \rmA)$,  where 
$1 \in \rmh^{0, 0}(\rmE\times _{\rmG}\rmY, \rmA)$ is the unit of the graded ring  $\rmh^{*, \bullet}(\rmE\times _{\rmG}\rmY, \rmA)$.
\vskip .1cm
In particular, $\pi_{\rmY}^*:\rmh^{*, \bullet}(\rmE\times _{\rmG}\rmY, \rmM) \ra \rmh^{*, \bullet}(\rmE\times _{\rmG}(\rmY \times \rmX) , \rmM)$ is split injective
if $\tr(\rmf_{\rmY})^*(1)= \tr(\rmf_{\rmY})^*(\pi_{\rmY}^*(1))$ is a unit, where $1 \in \rmh^{0, 0}(\rmE\times _{\rmG}\rmY, \rmA)$ is again the unit of the graded ring  $\rmh^{*, \bullet}(\rmE\times _{\rmG}\rmY, \rmA)$.
\item 
The transfer $\tr(\rmf_{\rmY})$ is compatible with  respect to restriction to subgroups of a given group. It is also compatible with
respect to change of base fields. \footnote{We prove in Theorem ~\ref{base.change} that 
the transfer is compatible with a more general base-change which includes this as a special case.}
\item
The map $ {\tr(\rmf_{\rmY})}^*:  \rmh^{*, \bullet}(\EG^{\it gm}\times _{\rmG}(\rmY \times \rmX), \rmM) \ra  { \rm h}^{*, \bullet}(\EG^{\it gm}\times _{\rmG}\rmY, \rmM)$ is independent of the choice of a geometric classifying space that satisfies certain basic assumptions (as in ~\ref{indep.class.sp}: Proof of Theorem ~\ref{thm.1}), and depends only on $\rmX$, $\rmY$ and the $\group$-equivariant map $\rmf$.
\item
Assume the base field $k$ satisfies the finiteness conditions in ~\eqref{etale.finiteness.hyps}. Assume
$\cE$ (which belongs to $\Spt(\k_{\rm mot})$) is  $\ell$-complete, in the sense of Definition ~\ref{Zl.local}, 
for some prime $\ell \ne char(\k)$.  Let $\epsilon^*: \Spt(\k_{\rm mot}) \ra \Spt(\k_{et})$
denote the functor induced by the obvious map sites from the \'etale site of $k$ to the Nisnevich site of $k$. 
\vskip .1cm
If $ \epsilon^*(\cE \wedge \rmX_+)$ is dualizable in $\SH(\k_{et}, \epsilon^*(\cE))$, 
and $\rmA$ is a weak ring spectrum in $\Spt(\k_{et}, \epsilon^*(\cE))$ with $\rmM$ a weak-module spectrum over $\rmA$, then there exists a transfer 
$\tr(\rmf_{\rmY})$ in $\SH(\k_{et}, \epsilon^*(\cE))$ satisfying similar properties.
\item 
Assume the base field $k$ satisfies the finiteness conditions in  ~\eqref{etale.finiteness.hyps}. Let $\cE$ denote a  commutative ring spectrum
in $\Spt(\k_{\rm mot})$ which is $\ell$-complete. Then if $\cE \wedge \rmX_+$ 
 is dualizable in $\SH(\k, \cE)$, $\epsilon^*(\cE \wedge \rmX_+)$ is 
dualizable in $\SH(\k_{et}, \epsilon^*(\cE))$, 
the transfer map $\tr(\rmf_{\rmY})$ is compatible with \'etale realizations, and for groups $\rmG$ that are special, $\epsilon^*(\tr(\rmf_{\rmY})) = \tr(\epsilon^*(\rmf_{\rmY}))$.
\end{enumerate}
\end{theorem}
The above theorem is proven by first establishing various key properties of the transfer as in the following theorems. We will adopt the terminology and notation as in \cite[Terminology 8.5]{CJ23-T1}.
%Throughout this section (and for the remainder of the chapter) we will adopt the notational conventions in section 
%~\ref{notation.approx}.
\begin{proposition} (Naturality with respect to base-change and change of groups) \index{transfer and base-change} \index{transfer and change of groups}
 \label{base.change}
Let $\rmG$ denote a linear algebraic group over $\k$ and let $\rmX$, $\rmY$ denote smooth quasi-projective $\group$-schemes over $\k$ or unpointed simplicial presheaves on $\Sm_{\k}$ provided with  $\rmG$-actions. 
Let $\rmp: \rmE \ra \rmB$ denote a $\rmG$-torsor with $\rmE$ and $\rmB$ smooth quasi-projective schemes over $\k$, with $\rmB$ connected and
let  $\rmf: \rmX \ra \rmX$ denote a $\rmG$-equivariant map.
\vskip .1cm
Let $\rmG'$ denote a closed linear algebraic subgroup of $\rmG$, $\rmp': \rmE' \ra \rmB'$ a $\rmG'$-torsor with $\rmB'$ connected, and $\rmY'$ a $\rmG'$-quasi-projective scheme
over $\k$ or an unpointed simplicial sheaf $\Sm_{\k}$ provided with a $\rmG'$-action, so that it comes equipped with a
map $\rmY' \ra \rmY$ that is compatible with the $\rmG'$-action on $\rmY'$ and the $\rmG$-action on $\rmY$. Further, we assume that one 
is provided with a commutative square
\[\xymatrix{{\rmE'} \ar@<1ex>[r] \ar@<1ex>[d] & {\rmE} \ar@<1ex>[d]\\
            {\rmB'} \ar@<1ex>[r]& {\rmB}}
\]
compatible with the action of $\rmG'$ ($\rmG$) on $\rmE'$ ($\rmE$, \res).  
\vskip .1cm
\begin{enumerate}[\rm(i)]
 \item 
Then if ${\Sigma^{\infty}_{\T}} \rmX_+$ is dualizable in $\SH(\k)$, the  square
\[\xymatrix{{{\Sigma^{\infty}_{\T}} (\rmE'\times_{\rmG'}(\rmY' \times \rmX))_+} \ar@<1ex>[r] & {{\Sigma^{\infty}} _{\T}(\rmE{\underset {\rmG}  \times} (\rmY \times \rmX))_+} \\
{{\Sigma^{\infty}_{\T}} (\rmE'\times_{\rmG'}\rmY' )_+} \ar@<1ex>[u]^{\tr(\rmf_{\rmY'})} \ar@<1ex>[r] & {{\Sigma^{\infty}_{\T}} (\rmE\times_{\rmG}\rmY )_+}  \ar@<-1ex>[u]^{\tr(\rmf_{\rmY})}} \]

\[(\xymatrix{{\rmR\epsilon_*(\epsilon^*{\Sigma^{\infty}_{\T}}) \wedge \rmR \epsilon_* (\rmE'\times_{\rmG'}^{et} (\rmY' \times \rmX))_+} \ar@<1ex>[r] & {\rmR\epsilon_*(\epsilon^*{\Sigma^{\infty}_{\T}}) \wedge \rmR \epsilon_* (\rmE\times_{\rmG}^{et}(\rmY \times \rmX))_+}\\
  {\rmR\epsilon_*(\epsilon^*{\Sigma^{\infty}_{\T}}) \wedge \rmR \epsilon_* (\rmE'\times_{\rmG'}^{et} \rmY')_+} \ar@<1ex>[u]^{\tr(\rmf_{\rmY'})} \ar@<1ex>[r] & {\rmR\epsilon_*(\epsilon^*{\Sigma^{\infty}_{\T}}) \wedge \rmR\epsilon_* (\rmE\times_{\rmG}^{et}\rmY)_+} \ar@<1ex>[u]^{\tr(\rmf_{\rmY})} )}\]
\vskip .2cm \noindent
commutes in the motivic stable homotopy category when $\rmG$ is special ($\rmG$ is not necessarily special, \res). 
\item
Next let $\cE$ denote a commutative ring spectrum in $\Spt(\k_{\rm mot})$  with $\cE \wedge X_+$ dualizable in $\SH(\k_{\rm mot}, \cE)$.  Then the square 
\[\xymatrix{{\cE \wedge (\rmE'\times_{\rmG'}(\rmY' \times \rmX))_+} \ar@<1ex>[r] & {\cE \wedge (\rmE \times_{\rmG} (\rmY \times \rmX))_+}  \\
{\cE \wedge (\rmE'\times_{\rmG'}\rmY' )_+} \ar@<1ex>[u]^{\tr(\rmf_{\rmY'})} \ar@<1ex>[r] & {\cE \wedge (\rmE\times_{\rmG}\rmY )_+}  \ar@<-1ex>[u]^{\tr(\rmf_{\rmY})} }\]

\[\xymatrix{({\rmR\epsilon_*(\epsilon^*\cE) \wedge \rmR\epsilon_*(\rmE'\times_{\rmG'}^{et} (\rmY' \times \rmX))_+} \ar@<1ex>[r] & { \rmR\epsilon_*(\epsilon^*\cE) \wedge \rmR\epsilon_*(\rmE\times_{\rmG}^{et}(\rmY \times \rmX))_+}\\
  { \rmR\epsilon_*(\epsilon^*\cE) \wedge \rmR\epsilon_*(\rmE'\times_{\rmG'}^{et} \rmY')_+} \ar@<1ex>[u]^{\tr(\rmf_{\rmY'})} \ar@<1ex>[r] & {\rmR\epsilon_*(\epsilon^* \cE) \wedge \rmR\epsilon_*(\rmE\times_{\rmG}^{et} \rmY)_+} \ar@<1ex>[u]^{\tr(\rmf_{\rmY})} )}\]

\vskip .2cm \noindent
commutes in the motivic stable homotopy category $\SH(\k, {\cE})$ when $\rmG$ is special ($\rmG$ is not necessarily special, \res). (In this case we may require 
$\ell$ is a prime
$ \ne char(\k)$ so that $\cE$ is $Z_{(\ell)}$-local or that $\cE$ is $\ell$-complete.)
\item
In case $\cE$ denotes a commutative ring spectrum in $\Spt(\k_{et})$ which is $\ell$-complete for some prime $\ell$ and $\cE \wedge \rmX_+$ is dualizable in $\Spt(\k_{et}, \cE)$,
one obtains a homotopy commutative diagram as in the first square in (ii) (with all the objects there replaced by their pull-backs to the big \'etale site) 
in the corresponding \'etale stable homotopy category $\SH(\k_{et}, \cE)$.
\end{enumerate}
\end{proposition}
\begin{proof} We will discuss this explicitly only in the case $\rmG$ is special, since the other case follows along similar lines with the obvious modifications. For each fixed representation $\rmV$ of $\rmG$, let $\xi^{\rmV}$ ($\eta^{\rmV}$) denote the vector bundles on $\widetilde {\rmB}$ chosen as in \cite[8.4, Step 2]{CJ23-T1}  (\cite[(8.4.7)]{CJ23-T1}, \res). Let
${\xi'}^{\rmV}$ (${\eta'}^{\rmV}$) denote the pull-back of these bundles to $\widetilde {\rmB}'$. Since $\xi^{\rmV} \oplus \eta^{\rmV}$ is trivial,
so is ${\xi'}^{\rmV} \oplus {\eta'}^{\rmV}$. Now the required property follows readily in view of this observation and the definition of the
transfer as in \cite[Definition 8.8]{CJ23-T1}, in view of the cartesian square, where $\pi_{\rmY}$ ($\pi_{\rmY'}$) is induced by the
projection $\rmY \times \rmX \ra \rmY$ ($\rmY' \times \rmX \ra \rmY'$, \res):
\be \begin{equation}
  \label{base.change.0}
 \xymatrix{{\rmE'\times_{\rmG'}(\rmY' \times \rmX)} \ar@<1ex>[r] \ar@<1ex>[d]^{\pi_{\rmY'}}& {\rmE\times_{\rmG}(\rmY \times \rmX)} \ar@<1ex>[d]^{\pi_{\rmY}}\\
             {\rmE'\times_{\rmG'}\rmY'} \ar@<1ex>[r] & {\rmE\times_{\rmG}\rmY}.}
\end{equation} \ee
The main observation here is that the diagrams in \cite[(8.4.3)]{CJ23-T1} through \cite[(8.4.17)]{CJ23-T1} for  ${\widetilde {\rmE}}'\times_{\rmG '}\rmY'$ map to the corresponding diagrams for ${\widetilde {\rmE}}\times_{\rmG}\rmY$ making the resulting diagrams commute, since
 all the transfers are constructed from the pre-transfers $\tr_{\rmG}(\rmf)'$ and $\tr_{\rmG'}(\rmf)'$.
\end{proof}
\begin{remark}
\label{change.groups}
Taking different choices for $\rmp'$ and $\rmY'$ provides many examples where the last Proposition applies. 
For example, let $\rmB' \ra \rmB$
denote a map from another smooth quasi-projective scheme that is also connected, let $\rmE' = \rmE{\underset {\rmB} \times} \rmB'$ and let $\pi_{\rmY}': \rmE'\times_{\rmG} (\rmY \times \rmX) \ra \rmE'\times_{\rmG} \rmY$ denote the
 induced map: this corresponds to base-change. (In particular, $\rmB'$ could be given by $\rmE/\rmH \cong \rmE\times_{\rmG} \rmG/\rmH$ for a closed subgroup $\rmH$ so that $\rmE/\rmH$ is connected.)
Moreover, the above proposition readily provides the following key multiplicative property of the transfer. 
\end{remark}
\begin{proposition} (Multiplicative property)
 \label{multiplicative.prop.0} 
\begin{enumerate}[\rm(i)]
\item   
Let $\rmH$ denote a linear algebraic group. 
%and let ${\group}= \rmH \times \rmH$ with $i = \Delta$ = the diagonal imbedding of $\rmH$ in ${\group}= \rmH \times \rmH$.
Let $\rmX$, $\rmY$ denote smooth quasi-projective schemes over $\k$ or unpointed simplicial presheaves on $\Sm_{\k}$ provided with  $\rmH$-actions, so that ${\Sigma^{\infty}_{\T}} \rmX_+$ is a dualizable
in $\SH(\k)$.
Let $\rmp: \rmE \ra \rmB$ denote an $\rmH$-torsor and let $\pi_{\rmY}:\rmE{\underset {\rmH} \times} \rmX \ra \rmB$ denote the induced map. Then the diagram 
\[\xymatrix{{{\Sigma^{\infty}_{\T}} \rmE\times _{\rmH}(\rmY \times \rmX)_+} \ar@<1ex>[r]^(.4){\rmd} & {{\Sigma^{\infty}_{\T}} (\rmE \times \rmE){\underset {\rmH \times \rmH} \times} ((\rmY \times \rmX) \times (\rmY \times \rmX))_+} \ar@<1ex>[r]^{{\rm id} \wedge {\rmq} \wedge {\rm id}} & {{\Sigma^{\infty}_{\T}} (\rmE \times \rmE){\underset {\rmH \times \rmH} \times}(\rmY \times(\rmY \times \rmX))_+} \\
{{\Sigma^{\infty}_{\T}} \rmE\times _{\rmH}\rmY_+} \ar@<1ex>[u]^{\tr(\rmf_{\rmY})} \ar@<1ex>[rr]^{\rmd} && {{\Sigma^{\infty}_{\T}} (\rmE\times _{\rmH}\rmY) _+ \wedge (\rmE\times _{\rmH}\rmY )_+} \ar@<1ex>[u]^{{\rm id} \wedge \tr(\rmf_{\rmY})}}\]
\vskip .2cm \noindent
commutes in $\SH(\k)$, when $\rmH$ is special. Here $\rmd$ denotes the  diagonal map induced by the diagonal map 
$\rmY \times \rmX \ra (\rmY \times \rmX) \times (\rmY \times \rmX)$ and  $\rmq$ denotes the map induced by the projection  $\rmY \times \rmX  \ra \rmY$. In case $\rmH$ is not necessarily special, one obtains
 a corresponding commutative diagram which is  obtained by replacing any term of the form $\Sigma^{\infty}_{\rmT}\rmE{\underset {\rmH} \times} \rmZ_+$ with
$ \rmR\epsilon_*(\epsilon^*\mbS_{\k}) \wedge \rmR\epsilon_*(\rmE{\underset {\rmH} \times}^{et} \epsilon^*(\rmZ))_+$.
 
\item  
In case $\cE$ is a commutative ring spectrum in $\Spt(\k_{\rm mot})$  with $\cE \wedge \rmX_+$ dualizable in $\SH(\k_{\rm mot}, \cE)$ and $\rmX$ is as in (i), the square
\[\xymatrix{{\cE \wedge \rmE\times _{\rmH}(\rmY \times \rmX)_+} \ar@<1ex>[r]^(.4){\rmd} & {\cE \wedge (\rmE \times \rmE){\underset {\rmH \times \rmH} \times} ((\rmY \times \rmX) \times (\rmY \times \rmX))_+} \ar@<1ex>[r]^{{\rm id} \wedge {\rmq} \wedge {\rm id}} & {\cE \wedge (\rmE \times \rmE){\underset {\rmH \times \rmH} \times}(\rmY \times(\rmY \times \rmX))_+} \\
{\cE \wedge \rmE\times _{\rmH}\rmY_+} \ar@<1ex>[u]^{\tr(\rmf_{\rmY})} \ar@<1ex>[rr]^{\rmd} && {\cE \wedge (\rmE\times _{\rmH}\rmY) _+ \wedge (\rmE\times _{\rmH}\rmY )_+} \ar@<1ex>[u]^{{\rm id} \wedge \tr(\rmf_{\rmY})}}\]
\vskip .2cm \noindent
also commutes $\SH(\k_{\rm mot}, {\cE})$,  when $\rmH$ is special. (Here we may also require that 
$\ell$ is a prime $ \ne char(\k)$ so that $\cE$ is $Z_{(\ell)}$-local or that $\cE$ is $\ell$-complete.) In case $\rmH$ is not necessarily special, one obtains
 a corresponding commutative diagram with all the terms above replaced as in (i).
 \item
  In case $\cE$ denotes a commutative ring spectrum in $\Spt(\k_{et})$ which
is $\ell$-complete for some prime $\ell \ne char(k)$, and $\cE \wedge \rmX_+$ is dualizable in $\Spt(\k_{et}, \cE)$,
the corresponding square commutes in the  \'etale stable homotopy category $\SH(\k_{et}, \cE)$.
\end{enumerate}
\end{proposition}
\begin{proof} We apply Proposition ~\ref{base.change} with the following choices: 
\begin{enumerate}[\rm(i)]
\item for $\rmG$, we take $\rmH \times \rmH$, for $\rmG'$ we take the diagonal $\rmH$ in $\rmH \times \rmH$,
\item for $\rmp$ ($\rmp'$) we take $\rmp \times \rmp: \rmE \times \rmE \ra \rmB \times \rmB$ (the given $\rmp: \rmE \ra \rmB$, \res),
\item for $\rmY$ we take $\rmY \times \rmY$ provided with the obvious action of $\rmH \times \rmH$ and for $\rmY'$ we take $\rmY$ provided
with the given action of $\rmH$.
\end{enumerate}
\vskip .1cm
Then we obtain the cartesian square as in ~\eqref{base.change.0} and the map in the corresponding top row
is given by 
\vskip .1cm
$ \rmE\times _{\rmH}(\rmY \times \rmX)  \ra  (\rmE \times \rmE){\underset {\rmH \times \rmH} \times} (\rmY  \times (\rmY \times \rmX))$.
\vskip .1cm \noindent
This factors as $ \rmE\times _{\rmH}(\rmY \times \rmX) {\overset d \ra}  (\rmE \times \rmE){\underset {\rmH \times \rmH} \times} ((\rmY \times \rmX) \times (\rmY \times \rmX)) {\overset {id \times q \times id} \ra}
(\rmE \times \rmE){\underset {\rmH \times \rmH} \times} (\rmY  \times (\rmY \times \rmX))$.
Therefore, Proposition ~\ref{base.change} applies.
\end{proof}
Next one may recall the definition of weak-module spectra from Definition ~\ref{weak.ring.modules}.
 
\begin{corollary} (Multiplicative property of transfer in generalized cohomology theories) \index{transfer and multiplicative property}
 \label{multiplicative.prop.1}
 Let $\rmH$ denote a linear algebraic group. 
 \begin{enumerate}[\rm(i)]
\item
Assume that $\rmA$ is a weak ring spectrum in $\Spt(\k_{\rm mot}, \cE)$, and that $\rmM$ is a weak module spectrum over $\rmA$. 
Assume that  $\rmX$, $\rmY$ are smooth quasi-projective schemes over $k$ or unpointed simplicial presheaves on $\Sm_{\k}$ provided with an action by $\rmH$ and that 
$\cE \wedge \rmX_+$ is dualizable in $\SH(\k_{\rm mot}, \cE)$. 
\vskip .1cm
Let $\pi_{\rmY}:  \rmE\times _{\rmH}(\rmY \times \rmX) \ra  \rmE\times _{\rmH}\rmY $ ($\pi_{\rmY}: \rmR\epsilon_*(\rmE \times _{\rmH}(\rmY \times \rmX) \ra \rmR \epsilon_*(\rmE \times _{\rmH}\rmY)$) denote the (obvious) map induced by the structure 
map $\rmX \ra Spec \, \k$ when $\rmH$ is special (when $\rmH$ is not necessarily special, \res).
\vskip .2cm
Then 
\[\tr(\rmf_{\rmY})^*(\pi_{\rmY}^*(\alpha). \beta) = \alpha .\tr(\rmf_{\rmY})^*(\beta), \quad \alpha \in \rmh^{*, \bullet}(\rmE\times _{\rmH}\rmY, \rmM), \beta \in \rmh^{*, \bullet}(\rmE\times _{\rmH}(\rmY \times \rmX)) , \rmA), \]
irrespective of whether $\rmH$ is special or not, and where the generalized motivic cohomology is defined
as in Definition ~\ref{coh.not}.
Here $\tr(\rmf_{\rmY})^*$ ($\pi_{\rmY}^*$) denotes the map induced on generalized cohomology by the map $\tr(\rmf_{\rmY})$ ($\pi_{\rmY}$, \res). In particular, 
\[\pi_{\rmY}^*: \rmh^{*, \bullet}(\rmE \times_{\rmH}\rmY, \rmM) \ra \rmh^{*, \bullet}(\rmE \times_{\rmH}(\rmY \times \rmX), \rmM) \]
 is {\it split injective} when $\rmH$ if $\tr(\rmf_{\rmY})^*(1)=\tr(\rmf_{\rmY})^*(\pi_{\rmY}^*(1)) $ is a unit, where $1 \in \rmh^{0, 0}(\rmE \times_{\rmH}\rmY, \rmA)$ is the unit of the graded ring $\rmh^{*, \bullet}(\rmE \times_{\rmH}\rmY, \rmA)$. 
\item
Assume that $\ell$ is a prime
$ \ne char(k)$ so that $\cE$ is $\ell$-complete, with $\epsilon^*(\cE \wedge \rmX_+)$ dualizable in $\SH(\k_{et}, \epsilon^*(\cE))$.
Assume further that 
$\rmA$ is a weak ring spectrum in $\Spt(\k_{et}, \epsilon^*(\cE))$, and that $\rmM$ is a weak module spectrum over $\rmA$. Then the 
same conclusions as in (i) hold for the generalized cohomology with respect to $\rmA$ and $\rmM$.
\end{enumerate}
\end{corollary}
\begin{proof} The proof of both statements  follow by applying the cohomology theory $\rmh^{*, \bullet}$ to all terms in the commutative diagram  in 
Proposition  ~\ref{multiplicative.prop.0}. We will provide details only for the case $\rmH$ is special, as the general case
 follows similarly. Moreover, as (ii) may be proven in the same manner as (i), 
  we will only discuss the proof of (i). 
 \vskip .2cm
In the proof of (i), one needs to start with $\rmh^{*, \bullet}(\rmE \times_{\rmH}\rmY, \rmM)$ and $\rmh^{*, \bullet}(\rmE\times_{\rmH} (\rmY \times \rmX), \rmM)$ along with
 the pairings:
\[\rmh^{*, \bullet}(\rmE \times_{\rmH}\rmY, \rmM) \otimes \rmh^{*, \bullet}(\rmE\times_{\rmH}(\rmY \times  \rmX), \rmA) {\overset {\pi_{\rmY}^* \otimes id} \ra } \rmh^{*, \bullet}(\rmE\times_{\rmH} (\rmY \times\rmX), \rmM) \otimes \rmh^{*, \bullet}(\rmE \times_{\rmH} (\rmY \times \rmX), \rmA) \ra
\rmh^{*, \bullet}(\rmE \times_{\rmH}(\rmY \times \rmX), \rmM), \]
\[\rmh^{*, \bullet}(\rmE \times_{\rmH}\rmY, \rmM) \otimes \rmh^{*, \bullet}(\rmE \times_{\rmH}\rmY, \rmA) \ra \rmh^*(\rmE \times_{\rmH}\rmY, \rmM). \]
This provides us with the commutative diagram:
\fontsize{9}{12}
\be \begin{equation}
     \label{module.prop}
\xymatrix{{\rmh^{*, \bullet}(\rmE \times_{\rmH}\rmY, \rmM) \otimes \rmh^{*, \bullet}(\rmE\times_{\rmH}(\rmY \times \rmX), \rmA)} \ar@<1ex>[rr]^(.5){id \otimes \tr(\rmf_{\rmY})^*} \ar@<1ex>[d]^{\pi_{\rmY}^* \otimes id} && {\rmh^{*, \bullet}(\rmE \times_{\rmH}\rmY, \rmM) \otimes \rmh^*(\rmE \times_{\rmH}\rmY, \rmA)} \ar@<1ex>[d]^{d^*} \\
 {\rmh^{*, \bullet}(\rmE \times_{\rmH} (\rmY \times \rmX), \rmM) \otimes \rmh^{*, \bullet}(\rmE \times _{\rmH}(\rmY \times \rmX), \rmA)} \ar@<1ex>[r]& {\rmh^{*, \bullet}(\rmE\times_{\rmH}(\rmY \times \rmX), \rmM)} \ar@<1ex>[r]^{\tr(\rmf_{\rmY})^*} & {\rmh^{*, \bullet}(\rmE \times_{\rmH}\rmY, \rmM)}}
\end{equation} \ee
\normalsize
Since the multiplicative property is the key to obtaining  splittings in the motivic stable homotopy category (see Theorems ~\ref{stable.splitting.1} and ~\ref{stable.splitting.2}), we will provide details on how one deduces commutativity of the 
above diagram. The commutativity of the above diagram follows
from the commutativity of a large diagram which we break up into three squares as follows, where $\rmRHom$ denotes
the derived external Hom in the category $\Spt(\k_{\rm mot}, \cE)$.
\vskip .2cm
\be \begin{equation}
     \xymatrix{{\rmRHom(\rmE \times_{\rmH} \rmY, \rmM) \otimes \rmRHom(\rmE \times_{\rmH}(\rmY \times \rmX), \rmA)} \ar@<1ex>[r] \ar@<1ex>[d]^{id \otimes \tr(\rmf_{\rmY})^*} & {\rmRHom((\rmE \times \rmE) _{\rmH \times \rmH} (\rmY \times (\rmY \times \rmX)), \rmM \wedge \rmA)} \ar@<1ex>[d]^{(id \wedge \tr(\rmf_{\rmY}))^*} \\
               {\rmRHom(\rmE \times_{\rmH}\rmY, \rmM) \otimes \rmRHom(\rmE \times_{\rmH}\rmY, \rmA)} \ar@<1ex>[r]& {\rmRHom((\rmE \times_{\rmH}\rmY)_+ \wedge (\rmE \times_{\rmH}\rmY)_+, \rmM \wedge \rmA)}}
 \end{equation} \ee
\[\xymatrix{{\rmRHom((\rmE \times \rmE) _{\rmH \times \rmH} (\rmY  \times (\rmY \times \rmX)), \rmM \wedge \rmA)} \ar@<1ex>[d]^{(id \wedge \tr(\rmf_{\rmY}))^*} \ar@<1ex>[r]^(.6){d^* \circ (\pi_{\rmY}^* \wedge id)}  &  {\rmRHom(\rmE\times_{\rmH}(\rmY \times \rmX), \rmM \wedge \rmA)} \ar@<1ex>[d]^{\tr(\rmf_{\rmY})^*} \\
          {\rmRHom((\rmE \times_{\rmH}\rmY)_+ \wedge (\rmE \times_{\rmH}\rmY)_+, \rmM \wedge \rmA)} \ar@<1ex>[r]^{d^*} & {\rmRHom(\rmE \times_{\rmH}\rmY, \rmM \wedge \rmA)}}\]
%{RHom(\EH \times \EH _{\rmH \times \rmH} \rmX \times \rmX, \rmM \wedge \rmA)} \ar@<1ex>[r]^{d^*} & {RHom(\EH \times_{\rmH} Spec\, \k, \rmM \wedge \rmA)} \ar@<1ex>[d]^{\tr(\rmf_{\rmY})^*}
%\vskip .2cm
\[\xymatrix{{\rmRHom(\rmE \times_{\rmH}(\rmY \times \rmX), \rmM \wedge \rmA)} \ar@<1ex>[d]^{\tr(\rmf_{\rmY})^*} \ar@<1ex>[r]^{\mu} & {\rmRHom(\rmE \times _{\rmH} (\rmY \times \rmX), \rmM)} \ar@<1ex>[d]^{\tr(\rmf_{\rmY})^*}\\
          {\rmRHom(\rmE \times_{\rmH}\rmY, \rmM \wedge \rmA)} \ar@<1ex>[r]^{\mu} & {\rmRHom(\rmE \times_{\rmH}\rmY, \rmM)}}\]
%\vskip .2cm
The commutativity of the first square is clear from the observation that the transfer $\tr(\rmf_{\rmY})$ is a stable map.
The commutativity of the second square is essentially the multiplicative property proved in 
Proposition ~\ref{multiplicative.prop.0}. The map $\mu$ in the third square is the map induced by the
pairing $\rmM \wedge \rmA \ra \rmM$. The commutativity of this square again follows readily from the 
observation that $\tr(\rmf_{\rmY})$ is a stable map. The commutativity of the square in ~\eqref{module.prop} results
 by composing the appropriate maps in the first square followed by the second and then the third square. More precisely, one can see that the
composition of the maps in the top rows of the three squares followed by the right vertical map in the last square equals 
$\tr(\rmf_{\rmY})^*  \circ d^* \circ (\pi_{\rmY}^* \otimes id)$ which is the composition of the left vertical map and the bottom row in the square ~\eqref{module.prop}.
Similarly the composition of the left vertical map in the first square and the maps in the bottom rows of the three squares above equals 
$ d^*  \circ (id \otimes \tr(\rmf_{\rmY})^*)$ which is the composition of the top row and the right vertical map in ~\eqref{module.prop}. This completes
the proof of the first statement in (i).
\vskip .2cm 
The last statement in (i) follows by taking 
$\beta =1 = \pi_{\rmY}^*(1) \in \rmh^{0,0}(\rmE\times _{\rmH}(\rmY \times \rmX), \rmA)$. The statement in (ii) follows from an entirely similar argument in the \'etale case.
\end{proof} 
\vskip .2cm
{ Next, we proceed to interpret the condition that $\tr(\rmf_{\rmY})^*(1)$ be a unit in $\rmh^{0,0}(\rmE\times_{\rmG}\rmY, \rmA)$, for a weak ring spectrum $\rmA$
in terms of the trace $\tau_{\rmX}(\rmf)$.}  We will do this only under the assumption the group $\rmG$ is also {\it special}. \index{transfer and trace}
Recall that we have a standing assumption that $\rmB$ is {\it connected}. 
In this context, we will {\it assume} the following:
$\rmY $ \mbox{ is a connected smooth scheme}  and  $\rmE\times_{\rmG} \rmY= \colimn\{\rmE_n\times_{\rmG} \rmY|n \ge 1\}$, where
$\{\rmE_n|n \ge 1\}$ 
is an ind-scheme with each  $ \rmE_n\times_{\rmG} \rmY$  is a connected smooth scheme so that   each $\rmB_n = \rmE_n/\rmG$
for $n>>0$  has  $\k$-{rational points}. We let $\rmB = \limn \rmB_n$.

We already observed in \cite[8.2.4. Convention]{CJ23-T1} that when $\rmB$ denotes
a finite degree approximation, $\BG^{\it gm,m}$, (with $m$ sufficiently large) of the classifying spaces for a linear algebraic group $\rmG$,
  it has $k$-rational points and is therefore a { geometrically connected smooth scheme of finite type over $\k$: see \cite[Tome 24, Chapitre 4, Corollaire 4.5.14]{EGA}.
Furthermore, we will assume that 
the generalized cohomology theory $\rmh^{*, \bullet}( \quad, \rmA)$ (defined with respect to the weak ring spectrum $\rmA$ as in Definition ~\ref{coh.not}) is such that the restriction 
\be \begin{align}
     \label{rat.point}
 \rmh^{0, 0}(\rmE \times_{\rmG}\rmY, \rmA) &\ra  \rmh^{0,0}(\rmY, \rmA)
%\rmh^{0, 0}(R\epsilon_*(\rmE \times_{\rmG}^{et}(a \circ \epsilon^*)(\rmY)), \rmA) &\ra  \rmh^{0,0}(R\epsilon_*((a \circ \epsilon^*)(\rmY_{\k'})), \rmA), \mbox{ for } \rmG \mbox{ not special} \notag
\end{align} \ee
is an isomorphism, where ${\rm Spec} \, \k \ra \rmB$ is any $k$-rational point of $\rmB$, and $\rmY$ is viewed as the fiber of $\rmE\times_{\rmG} \rmY$ over ${\rm Spec } \, \k$.
\begin{proposition}
 \label{trace.2}
 Assume that both $\rmY$ and $\rmX$ are smooth schemes of finite type over $\k$ and 
 $\rmG$ is special. % or (ii) $\rmE \times _{\rmG} \rmY$ is ${\mathbb A}^1$-fibrant.
\vskip .1cm
 Under the above assumptions %as well as ~\eqref{rat.point}, 
 \[\tr(\rmf_{\rmY})^*(1) = \tr(\rmf_{\rmY})^* (\pi_{\rmY}^*(1)) = (id_{\rmY} \wedge \tau_{\rmX}(f))^*(1)\]
 where $\tau_{\rmX}(\rmf)$ is the trace defined in \cite[Definition 8.4(iv)]{CJ23-T1}
and $\pi_{\rmY}^*:\rmh^{*, \bullet}(\rmE \times_{\rmG}\rmY, \rmA) \ra \rmh^{*, \bullet}(\rmE\times_{\rmG}(\rmY \times \rmX), \rmA)$, 
$\tr(\rmf_{\rmY})^*:\rmh^{*, \bullet}(\rmE\times_{\rmG}(\rmY \times \rmX), \rmA) \ra \rmh^{*, \bullet}(\rmE \times_{\rmG}\rmY, \rmA)$ denote the
induced maps.
\end{proposition}
\begin{proof} We discuss explicitly only the case where $char (k)=0$. In positive characteristics $p$, one needs to replace the sphere
spectrum $\mbS_{\k}$ everywhere by the corresponding sphere spectrum with the prime $p$ inverted, or completed away from $p$. 
\vskip .1cm
%We will first assume $\rmG$ is special. 
Then the first equality is clear since $\pi_{\rmY}^*$ is a ring homomorphism and therefore, $\pi_{\rmY}^*(1)=1$.
 Recall that we are assuming the group $\rmG$ is special. The naturality with respect to base-change as in Proposition ~\ref{base.change}, together with the assumption that the restriction $\rmh^{0,0}(\rmE \times_{\rmG}^{et}\rmY, \rmA) \ra \rmh^{0,0}(\rmY, \rmA)$
is an isomorphism shows that $\tr(\rmf_{\rmY})^*\pi_{\rmY}^*(1)$ is the same for $\rmE \times_{\rmG}\rmY$ as well as for $\rmY$.
When $\rmG$ is special and the scheme $\rmB = Spec \, \k$, $\tr(\rmf): \cE\wedge ({\rm Spec }\, \k)_+ \ra \cE \wedge \rmX_+$, (which also identifies
 with the corresponding pre-transfer $\tr(\rmf)'$) so that, $\tr(\rmf_{\rmY}) =id_{\rmY} \wedge \tr(\rmf)'$. Therefore, it is clear
 that $\pi_{\rmY} \circ \tr(\rmf_{\rmY}) = id_{\rmY} \wedge \tau_{\rmX}(f)$ as defined in \cite[Definition 8.4 (iii), (iv)]{CJ23-T1}. 
 Therefore, the equality $\tr(\rmf_{\rmY})^* (\pi_{\rmY}^*(1)) = (id_{\rmY} \wedge \tau_{\rmX}(f))^*(1)$ follows, and completes the proof.

\end{proof}
\begin{proposition}
	\label{no.need.special}
%\begin{enumerate}[\rm(i)]
%\item 
The hypotheses in ~\eqref{rat.point} are satisfied when $\rmh^{*, \bullet}$ denotes motivic cohomology with respect to any commutative ring $\rmR$,
 $\rmY$ is a connected smooth scheme,   
$\rmE\times_{\rmG}^{et} \rmY= \colimn\{\rmE_n\times_{\rmG} \rmY|n \ge 1\}$, where
$\{\rmE_n|n \ge 1\}$ 
is an ind-scheme with each  $ \rmE_n\times_{\rmG} \rmY$  a connected smooth scheme, and when $\rmB_n = \rmE_n/\rmG$, for $n>>0$ has
$\k$-rational points. 
%\item
%Moreover, $\rmY= {\rm Spec} \, \k$,  $\rmY= Spec \, \k'$ with $\k'$ a finite Galois
% extension of $\k$ of order prime to the characteristic,  
%$\colimn{\rm EG}^{\it gm,} \times_{\rmG} {\rm Spec } \,\k = \colimn{\rm BG}^{\it gm,n}$, and $\colimn{\rm EG}^{\it gm,n} \times_{\rmG} {\rm Spec } \,\k' $ with $\k'$ a finite Galois
% extension of order prime to the characteristic  are all ${\mathbb A}^1$-fibrant.
% \end{enumerate}
\end{proposition}
\begin{proof}
Recall that  $\rmY$ and $\rmE\times_{\rmG}\rmY$ are assumed to be connected.
Observe that now, $\rmh^{0,0} =\rmH_{\rmM}^{0,0}$, which denotes motivic cohomology in degree $0$ and weight $0$. The motivic complex
$\rmR(0)$ is the constant sheaf associated to the ring $\rmR$.
Therefore, since $\rmE\times_{\rmG}\rmY$ and $\rmY$ are connected, the restriction map 
$\rmH^{0,0}(\rmE\times_{\rmG}\rmY, \rmR) \ra \rmH^{0,0}(\rmY, \rmR)$ is an isomorphism, where $\rmY$ is the fiber of $\rmE\times_{\rmG}\rmY \ra \rmB$ over any $k$-rational point of $\rmB$. 
It follows that
the hypothesis in ~\eqref{rat.point} is always satisfied, when $\rmh^{*, \bullet}( \quad , \rmA)$ denotes motivic cohomology with respect to any commutative ring $\rmR$. This proves
the proposition.

\end{proof}
\vskip .2cm \noindent
\subsection{\bf Proof of Theorem ~\ref{thm.1}} 
\label{indep.class.sp}
We will first clarify the terminology used. Recall that ${\BG}^{\it gm,m}$ 
(${\EG}^{\it gm,m}$) denotes the $m$-th degree approximation to the classifying 
space of the group $\group$ (its principal $\group$-bundle, \res) as in \cite{MV}: see also \cite{Tot}. If $\rmX$ is a 
scheme with $\rmG$-action, one can form the scheme ${\EG}^{\it gm,m}\times_{\rmG} \rmX$, which is called
the Borel construction. In case $\rmG$ is not special, the torsor ${\EG}^{\it gm, m} \ra {\BG}^{\it gm, m}$ is locally trivial
 only in the \'etale topology, so that in this case we replace the Borel construction above by
 $\rmR\epsilon_*({\EG}^{\it gm,m}\times_{\rmG}^{et} \rmX)$ as discussed in \cite[section 8.3, Case 2]{CJ23-T1}.
 However, we will continue to denote $\rmR\epsilon_*({\EG}^{\it gm,m}\times_{\rmG}^{et} \rmX)$ by ${\EG}^{\it gm,m}\times_{\rmG} \rmX$ mainly for the sake of 
 simplicity of notation.
 \vskip .1cm
The first statement in the Theorem is the compatibility
of the transfer with various degrees of finite dimensional approximations to the classifying space: this has been discussed in \cite[8.4, Step 2]{CJ23-T1} in the construction of the transfer.
The second statement in the Theorem 
 is the multiplicative property
proven in Corollary ~\ref{multiplicative.prop.1}. Taking $\beta =1$ in (ii) proves that if $\tr(\rmf_{\rmY})^*(1)$ is a unit, then the composition
$\tr(\rmf_{\rmY})^*\circ \pi_{\rmY}^*$ is an isomorphism for any weak module spectrum $\rmM$ over $\rmA$. Conversely if $\tr(\rmf_{\rmY})^*\circ \pi_{\rmY}^*$
 is an isomorphism, then the multiplicative property in (ii) shows that, there exists some class $\alpha \in \rmh^{0, 0}(\rmE\times _{\rmG}\rmY, \rmA)$ 
 so that $\alpha. \tr(\rmf_{\rmY})^*(\pi_{\rmY}^*(1)) = 1 \in \rmh^{0, 0}(\rmE\times _{\rmG}\rmY, \rmA)$, which shows $\tr(\rmf_{\rmY})^*(\pi_{\rmY}^*(1))$
 is a unit in $\rmh^{0, 0}(\rmE\times _{\rmG}\rmY, \rmA)$. This proves the first statement in (iii).
\vskip .1cm
The second statement in (iii) is now clear, and the first statement in (iv) follows from the naturality property for the transfer with respect to
base-change as in Proposition ~\ref{base.change}. (See Remark ~\ref{change.groups}.) That the transfer
is compatible with change of base fields follows from the corresponding property for the
pre-transfer: see \cite[Proposition 7.3]{CJ23-T1}. The second statement in (ii) follows from the fact that the transfer is defined using the pre-transfer, which is a stable map
that involves no degree or weight shifts. 
\vskip .1cm
Next we will sketch an argument to prove Theorem ~\ref{thm.1}(v).
 Let $\{\BG^{\it gm,m}(1)|m\}$, $\{\BG^{\it gm,m}(2)|m \}$ denote two sequences of
finite degree approximations to the classifying space of the given group $\group$
satisfying certain basic assumptions as in \cite{MV}. (See also \cite{Tot}.) Let $\{ \EG^{\it gm,m}(1), \EG^{\it gm,m}(2)|m\}$ denote the corresponding universal
$\group$-bundles: { the main requirements here are that both these have free actions by $\group$ and that as $m \ra \infty$, these are ${\mathbb A}^1$-acyclic}.
\vskip .1cm
Then a key observation is that $\{\EG^{\it gm,m}(1) \times \EG^{\it gm,m}(2)|m\}$ with the 
diagonal action of the group $\group$ also satisfies the same hypotheses so that their
quotient by the diagonal action of $\group$ will also define approximations to the 
classifying space of the $\group$. Therefore, after replacing $\{\EG^{\it gm,m}(1)|m\}$ with
$\{\EG^{\it gm,m}(1) \times \EG^{\it gm,m}(2)|m\}$, we may assume that one has a direct system
of smooth surjective maps $\{\EG^{\it gm,m}(1) \ra \EG^{\it gm,m}(2)|m\}$. Now it is straightforward to verify that all the constructions discussed in the above steps for
the transfer are compatible with the maps $\{\EG^{\it gm,m}(1) \ra \EG^{\it gm,m}(2)|m\}$. 
Therefore, by Proposition ~\ref{base.change}, one obtains a direct system of homotopy commutative diagrams, $m \ge 1$:
\[ \xymatrix{{{\Sigma^{\infty}_{\T}}(\EG^{\it gm,m}(1)\times_{\rmG}X)_+} \ar@<1ex>[r] & {{\Sigma^{\infty}_{\T}}(\EG^{\it gm,m}(2)\times_{\rmG}X)_+}\\
	{{\Sigma^{\infty}_{\T}}\BG^{\it gm,m}(1)_+} \ar@<1ex>[u]^{tr(f)^m(1)}\ar@<1ex>[r] & {{\Sigma^{\infty}_{\T}}\BG^{\it gm,m}(2)_+} \ar@<1ex>[u]^{tr(f)^m(2)} .}\]
Finally, one may also verify that the maps
$\{\EG^{\it gm,m}(1) \times_{\rmG}X \ra \EG^{\it gm,m}(2) \times_{\rmG}X |m\}$ and $\{\BG^{\it gm,m}(1) \ra \BG^{\it gm,m}(2)|m\}$ induce isomorphisms on  generalized motivic cohomology as one takes the $\colimm$: see, for example \cite[\S4, Proposition 2.6]{MV}.
These complete the proof of (v) when $\rmY = Spec \, \k$: the case when $\rmY$ is a general smooth $\rmG$-scheme or a simplicial presheaf with $\rmG$-action 
is similar.
\vskip .1cm
The construction
 of the transfer in the \'etale framework is entirely similar, though care has to be taken to ensure that affine spaces are contractible
 in this framework, which accounts partly for the hypothesis in (vi) and in ~\eqref{etale.finiteness.hyps}. Property (vii) is proved in
 the next section. \qed
  \section{\bf \'Etale and Betti realization}
 \label{et.real.1}
\subsection{\it The motivation for considering \'etale and Betti realization} We hope this short discussion clarifies the role
of \'etale and Betti realization in the context of the transfer. 
\vskip .1cm
We make use of the following {\it two  distinct strategies} for obtaining splittings using the 
transfer:
\vskip .1cm
$\bullet$ One of these is to show certain classes are units in the 
Grothendieck-Witt ring of the base field: this is quite difficult and is do-able only for very special cases as explained later: see the discussion following
the proof of Theorem ~\ref{split.0}.
\vskip .1cm
$\bullet$ The second technique is to restrict to what are slice-completed
generalized motivic cohomology theories (defined in Definition ~\ref{completed.gen.coh}): since several of the well-known generalized motivic cohomology theories, such as 
Algebraic K-Theory and Algebraic Cobordism are slice completed, this is not a major restriction.
\vskip .1cm
Since the transfer is a stable map, it induces a map of the motivic Atiyah-Hirzebruch spectral sequences 
for the above generalized cohomology theories. The $E_2$-terms of these spectral sequences are modules over the motivic cohomology
of the corresponding motivic spaces. By the multiplicative property of the transfer (see Corollary ~\ref{multiplicative.prop.1}), we may then reduce
to obtaining splittings at the level of motivic cohomology. 
\vskip .1cm
It is precisely at this point that it becomes quite convenient to know that the pre-transfer and transfer are compatible with both
\'etale and Betti realization (see section ~\ref{compatibility.real}), so that one may reduce to showing the transfer in \'etale cohomology or Betti cohomology produces
the desired splittings.  Clearly proving such splittings at the level of \'etale or Betti cohomology is considerably easier than showing this
at the level of motivic cohomology. Nevertheless, we show that such splittings at the level of the \'etale or Betti realization is enough
to show the desired splitting exists at the level of motivic cohomology: see Proposition ~\ref{splitting.et}. %section ~\ref{compatibility.real}. 
\vskip .1cm
In view of the above observations, our primary interest is in considering \'etale realization and Betti realization {\it only} with respect to 
 Eilenberg-Maclane spectra: this justifies the following short discussion. A more detailed discussion of \'etale realization in general
is left to the Appendix.
\vskip .2cm
\subsection{\it Passage to spectra on the \'etale site}

We start with the following morphism of sites: $\epsilon: {\rm Sm}_{\k, et} \ra {\rm Sm}_{\k, \Nis}$,
$\bar \epsilon:{\rm Sm}_{\bar \k, et} \ra {\rm Sm}_{\bar \k, \Nis}$ and $\eta: {\rm Sm}_{\bar \k, et}\ra {\rm Sm}_{\k, et}$, 
where $\bar \k$ denotes the algebraic closure of $\k$.
\vskip .2cm
In view of the discussion of the alternate model structure on Nisnevich presheaves defined by inverting hypercovers and the 
discussion of the model structure on \'etale presheaves as in \cite[2.1.7]{CJ23-T1}, one can see that 
these induce the following functors:
\be \begin{equation}
    \label{maps.topoi.0}
 \epsilon^*:{\Spt}^{\group}(\k_{\rm mot}) \ra {\Spt}^{\group}(\k_{et}),  \bar \epsilon^*:{\Spt}^{\group}({\bar \k}_{\rm mot}) \ra  {\Spt}^{\group}({\bar \k}_{et}),  \eta^*:  {\Spt}^{\group}({\k}_{et}) \ra  {\Spt}^{\group}({\bar \k}_{et}),\\
\end{equation} \ee
\be \begin{equation}
 \epsilon^*:{\widetilde {\Spt}}^{\group}(\k_{\rm mot}) \ra {\widetilde {\Spt}}^{\group}(\k_{et}),  \bar \epsilon^*:{\widetilde {\Spt}}^{\group}({\bar \k}_{\rm mot}) \ra {\widetilde {\Spt}}^{\group}({\bar \k}_{et}) \mbox{ and } \eta^*: {\widetilde {\Spt}}^{\group}({\k}_{et}) \ra {\widetilde {\Spt}}^{\group}({\bar \k}_{et}), \notag
 \end{equation}\ee
  \vskip .1cm \noindent
as well as corresponding functors for the categories ${\widetilde {\Spt}}(\k_{\rm mot})$, ${\widetilde {\Spt}}^{\group}(\k_{et})$, $\Spt(\k_{\rm mot})$, $\Spt(\k_{et})$ and similar categories of 
spectra defined in \cite[Definitions  4.6 through 4.9]{CJ23-T1}. 
Here the above categories are provided with the stable injective model structures as mentioned in \cite[{\rm Terminology}]{CJ23-T1}. Since we work with the injective model structures on
all of the above categories, it is clear that these are left Quillen functors.
\vskip .2cm
\begin{proposition}
 The above functors are weakly monoidal functors.
\end{proposition}
\begin{proof} We merely observe that this follows readily from the definition of the smash product as a co-end in \cite[Definition 4.2, 4.8(ii) and 4.9(iii)]{CJ23-T1}.
\end{proof}
\begin{definition}
 \label{et.real}
We call the functor $\epsilon ^*$ ($\bar \epsilon^*$), {\it \'etale realization over $\k$} ({\it \'etale realization over $\bar \k$}, \res). \qed
\end{definition} \index{\'etale realization}
\vskip .1cm \noindent
The justification for the above definition will be clear from Proposition ~\ref{et.real.Eilen.Mac}. 
\vskip .2cm
We proceed to consider \'etale Eilenberg-Maclane spectra. We start with the simplicial $2$-sphere $\rmS^2$ identified with $\Delta[2]/\delta \Delta[2]$.
Recall that the free ${\mathbb Z}/\ell$-module on the above $\rmS^2$ followed by the 
forgetful functor sending a simplicial abelian presheaf to the underlying simplicial presheaf provides the Eilenberg-Maclane space ${\rm K}({\mathbb Z}/\ell, 2)$. 
Therefore, one obtains a natural map $\rmS^2 \ra \rmK({\mathbb Z}/\ell,2)$.
 \begin{definition}
\label{Etale.EM.sp}
Let $\ell$ be different from the characteristic of the base field $k$ and $n$ a fixed positive integer.
We let the {\it \'etale Eilenberg-Maclane spectrum} denote the sheaf of $\rmS^2$-spectra defined as follows: the space in degree $2m$, for a non-negative integer $m$ is
the constant sheaf ${\rm K}({\mathbb Z}/{\ell}^n, 2m)$ (whose sheaf of homotopy groups are trivial in all degrees except $2m$, where it is ${\mathbb Z}/{\ell}^n$), and whose structure 
maps are defined by the pairing $\rmS^2 \wedge K({\mathbb Z}/{\ell}^n, 2m) \ra {\rm K}({\mathbb Z}/{\ell}^n, 2) \wedge {\rm K}({\mathbb Z}/{\ell}^n, 2m) \ra {\rm K}({\mathbb Z}/{\ell}^n, 2m+2)$.
\end{definition}
 Let $\rmS = {\rm Spec } \, \k$. Recall that the 
 Motivic Eilenberg-Maclane spectrum ${\mathbb H}({\mathbb Z}/\ell^n)_{\rmS}$ has as its $m$-th space $\U({\mathbb Z}/\ell^{n,tr}(\T^m))$, where 
 where $\U$ denotes the forgetful functor sending a simplicial abelian presheaf (with transfers) to the underlying simplicial presheaf.
Next we apply $\epsilon^*$ to the motivic Eilenberg-Maclane spectrum ${\mathbb H}({\mathbb Z}/\ell^n)_{\rmS}$ to obtain the following object: the sequence 
$\{\epsilon^*({\mathbb Z}/\ell^{n,tr}(\T^m))|m \ge 0\}$ with the structure maps
\be \begin{multline}
     \begin{split}
       \label{et.EM.sp}
\epsilon^*( \T) \wedge \U(\epsilon^*({\mathbb Z}/\ell^{n, tr}(\T^m))) \ra \U(\epsilon^*({\mathbb Z}/\ell^{n, tr}(\T))) \wedge \U(\epsilon^*({\mathbb Z}/\ell^{n, tr}(\T^m))) \ra  \\
\U(\epsilon^*({\mathbb Z}/\ell^{n, tr}(\T) \otimes^{tr} {\mathbb Z}/\ell^{n, tr}(\T^m))) \cong \U(\epsilon^*({\mathbb Z}/\ell^{n, tr}(\T^{m+1}))).
   \end{split}
 \end{multline} \ee

\begin{proposition}
 \label{et.real.Eilen.Mac}
 Let $\rmS = Spec \, \k$, where $k$ is a perfect field.
 Then the \'etale realization of the Motivic Eilenberg-Maclane spectrum ${\mathbb H}({\mathbb Z}/\ell^n)_{\rmS}$, where $\ell$ is different from 
 $char(k)$, is the spectrum whose $n$-th term is the presheaf of Eilenberg-Maclane spaces $\rmK(\mu_{\ell^n}(m), 2m)$ and whose structure maps
 are given by the pairings: 
 \be \begin{equation}
 \label{et.EM.sp.2}
 \epsilon^*(\T) \wedge \rmK(\mu_{\ell^n}(m), 2m) \ra \rmK(\mu_{\ell^n}(1), 2) \wedge \rmK(\mu_{\ell^n}(m), 2m) \ra \rmK(\mu_{\ell^n}(m+1), 2m+1).
 \end{equation} \ee
 \vskip .1cm \noindent
 Assume in addition that the base field $\k$ has a primitive $\ell^n$-th root of unity. Then the above spectrum identifies with the \'etale 
 Eilenberg-Maclane spectrum considered in Definition ~\ref{Etale.EM.sp}.
\end{proposition}
\begin{proof}
 A key observation here is the following:
 \be \begin{equation}
  \label{mot.et.Eilen.Mac}
 \U( \epsilon^*({\mathbb Z}/\ell^{n,tr}(\T^m))) \simeq \rmK(\mu_{\ell^n}(m), 2m),
     \end{equation} \ee
 This is discussed in \cite[Theorem 10.3 and Theorem 15.2]{mvw}.  Recall the structure maps of the spectrum ${\mathbb H}({\mathbb Z}/\ell^m)_{\rmS}$
are given by the pairing 
\[\U({\mathbb Z}/\ell^{n,tr}(\T ))\wedge \U({\mathbb Z}/\ell^{n,tr}(\T^m)) \ra \U({\mathbb Z}/\ell^{n,tr}(\T) \otimes^{tr} {\mathbb Z}/\ell^{n,tr}(\T^m)) \ra \U({\mathbb Z}/\ell^{n,tr}(\T^{m+1}).\] 
Therefore, the observation that $\U(\epsilon^*({\mathbb Z}/\ell^{n,tr}(\T))) \simeq {\rm K}(\mu_{\ell^n}(1), 2)$
 proves the first statement.
 \vskip .1cm
 Since $\k$ is assumed to have a primitive $\ell^n$-th root of unity, $\mu_{\ell^n}(1) \cong {\mathbb Z}/\ell^n$ and therefore
 $\rmK(\mu_{\ell^n}(1), 2) \cong \rmK({\mathbb Z}/\ell^n, 2)$. Similarly $\rmK(\mu_{\ell^n}(m), 2m)$ identifies with $\rmK({\mathbb Z}/\ell^n, 2m)$. This proves
 the second statement and completes the proof of the Proposition.
\end{proof}
\begin{remark}
\label{remark.et.real}
(Further remarks on \'etale realization)
\'Etale realization in {\it the unstable setting} has been discussed in \cite{Isak} and also \cite{Schm} where they  consider \'etale realization of 
motivic spaces to take values in the category of pro-simplicial sets, following the setting of \'etale homotopy theory as in \cite{AM}. So far there has been no discussion of \'etale realization of motivic spectra. A discussion of
\'etale realization for motivic spectra in general, i.e., apart from what is obtained by pull-back to the \'etale site 
and also apart from the special case of the Eilenberg-Maclane spectra is extraneous to our goals. 
 In fact the only place where we need to
 show that the transfer is compatible with \'etale realization is at the level of motivic and \'etale cohomology over algebraically closed fields and away from the 
 characteristic of the base field:  see Proposition ~\ref{transf.real}, Corollary ~\ref{transf.compat}, as well as Theorem ~\ref{stable.splitting.2} and
 Proposition ~\ref{splitting.et}. Therefore, 
 the above comparison results suffice to show that often splitting at the level of etale cohomology by means of the transfer implies splitting at the level of motivic cohomology.
 \vskip .1cm
 Nevertheless, in view of general interests in obtaining a stable version of \'etale realization, we provide a short discussion on this topic in the Appendix,
 though it is not used in the rest of the paper. \qed
\end{remark}
As is well-known, \'etale cohomology is well-behaved only with respect to torsion coefficients prime to the residue characteristics. This makes it necessary to consider
 completions away from the characteristic of the base field on considering spectra on the \'etale site. This justifies the following definitions.
\begin{definition} 
\label{Zl.local}
Let $\rmM \in \SH(\k)$ ($\SH(\k_{et}$). For each prime number $\ell$, let ${\mathbb Z}_{(\ell)}$  denote the localization of the integers at the prime ideal corresponding to $\ell$ and let ${\mathbb Z}\compl_{\ell} =\invlimn {\mathbb Z}/\ell^n$.
Then we say $\rmM$ is ${\mathbb Z}_{(\ell)}$-local ($\ell$-complete, $\ell$-primary torsion), if each $[{\rmS^1}^{\wedge s} \wedge \T^t \wedge {\Sigma^{\infty}_{\T}}\rmU_+, \rmM]$ is a ${\mathbb Z}_{(\ell)}$-module (${\mathbb Z}\compl_{\ell}$-module, ${\mathbb Z}\compl_{\ell}$-module which is torsion, \res) as $\rmU$ varies among the objects of the given site, where 
$[{\rmS^1}^{\wedge s}\wedge \T^t \wedge {\Sigma^{\infty}_{\T}}\rmU_+, \rmM]$
 denotes $Hom$ in the stable homotopy category $\SH(\k)$ ($\SH(\k_{et})$, \res). \index{$\ell$-complete spectra} \index{${\mathbb Z}_{(\ell)}$-local spectra}
\end{definition}
Let $\rmM \in \SH(\k)$ ($\SH(\k_{et})$). Then one may observe that if $\ell$ is a prime number, and $\rmM$ is $\ell$-complete, then $\rmM$ is ${\mathbb Z}_{(\ell)}$-local. This follows readily by observing that the natural map ${\mathbb Z} \ra {\mathbb Z}\compl_{\ell}$ factors through ${\mathbb Z}_{(\ell)}$ since
 every prime different from $\ell$ is inverted in ${\mathbb Z}\compl_{\ell}$. One may also observe that if $\cE$ is a commutative ring spectrum which
is ${\mathbb Z}_{(\ell)}$-local ($\ell$-complete), then any module spectrum $\rmM$ over $\cE$ is also ${\mathbb Z}_{(\ell)}$-local 
($\ell$-complete, \res). $\ell$-completion in the motivic framework is discussed in some detail in the Appendix.%: see also \cite[section 4]{CJ17}.
\vskip .2cm 
We conclude this section with the following remarks on Betti realization.
\subsection{\bf Betti realization}
 When the base scheme is the field of complex numbers, there is a fairly extensive discussion on Betti realization both in the unstable and stable settings: 
 see \cite{Ay},  %\cite{Lev13} 
 and \cite{PPR}. %In fact a result analogous to Proposition ~\ref{et.real.Eilen.Mac} is discussed in \cite[Proposition 5.6]{Lev13}. \index{Betti realization}
 %\vskip .2cm
 %{\it For the convenience of the reader, the remainder of this chapter will be devoted to an overview of the key constructions and results of the paper that are discussed in later chapters in full detail.}

\section{\bf Computing traces: compatibility of the transfer with realizations}
\label{compatibility.real}
 Assume the situation as in Theorem ~\ref{thm.1}. Then, very often the main application of the transfer is to prove that $\pi_{\rmY}^*$ is a split injection in generalized cohomology, i.e., one needs to verify that 
 $\tr(\rmf_{\rmY})^*(\pi_{\rmY}^*(1))$ is a unit.
In order to verify that
 $\tr(\rmf_{\rmY})^*(\pi_{\rmY}^*(1))$ is a unit, one may adopt the following strategy. First we will show that the transfer constructed above is compatible with
passage to a simpler situation, for example passage from over a given base field to its algebraic or separable closure and/or passage to 
a suitable {\it realization} functor: we will often use the \'etale realization. Then, often, $\rmh^{0, 0}(\rmB) \simeq \rmh^{0,0}_{real}(\rmB)$
where $\rmh^{*, \bullet}_{real}(\rmB)$ denotes the corresponding generalized cohomology of the realization and $\rmB$ denotes any smooth scheme. Therefore, it will
suffice to show that $\tr(\rmf_{\rmY})_{real}^*(\pi_{\rmY}^*(1))$ is a unit: here $\tr(\rmf_{\rmY})_{real}$ denotes the corresponding transfer on the realization.
We devote all of this section to a detailed discussion of this technique.
\vskip .2cm
{ As before we will assume the base scheme is the spectrum of a perfect field $k$ satisfying the assumption ~\eqref{etale.finiteness.hyps}. $\bar k$ will denote a fixed algebraic closure of $k$ and $\ell$ is a prime different from $char(k)$.}
Accordingly $\rmS= {\rm Spec} \, {\it k}$ and $\bar \rmS= {\rm Spec} \, \bar \k$.
We first recall the following functors (from ~\cite[(7.0.14)]{CJ23-T1}):
\be \begin{equation}
     \label{maps.topoi.3}
\epsilon^*:\Spt(\k_{\rm mot}) \ra \Spt(\k_{et}),  \bar \epsilon^*:\Spt(\bar \k_{\rm mot}) \ra \Spt(\bar \k_{et}),   \mbox{ and } \eta^*: \Spt(\k_{et}) \ra \Spt(\bar \k_{et}).
\end{equation}\ee
\vskip .2cm
Let $\theta$ and $\phi _{\cE}$ denote the functors considered in \cite[Proposition 7.3]{CJ23-T1}. We let $\cE \in \Spt(\k_{\rm mot})$
denote a commutative ring spectrum which is $\ell$-complete for a prime $\ell \ne char(k)$. 
%Throughout the following discussion, we will take $\rmY = Spec \, \k$.

\begin{proposition}(Commutativity of the pre-transfer with \'etale realization)
\label{transf.real}
Assume the above situation. 
Then denoting by $\tr'(\rmf)$ the pre-transfer (as in  \cite[Definition 8.2 (ii)]{CJ23-T1}) (with $\rmG$ trivial), 
$ \epsilon^*(\tr'(\rmf)) \simeq \tr'(\epsilon^*(\rmf))$ and $ \bar \epsilon^*(\tr'(\rmf)) \simeq \tr'( \bar \epsilon^*(\rmf))$
when applied to dualizable objects of the form $\cE \wedge X_+$ appearing in \cite[Theorem 7.1]{CJ23-T1}. The same conclusion holds for $\epsilon^*$ and $\bar \epsilon^*$ replaced by
$\eta^*$ or any of the two functors $\theta$ and $\phi _{\cE}$.
\end{proposition}
\begin{proof} Implicitly assumed in the proof is the fact that the above functors  all send dualizable objects to dualizable objects. This is
already proved in \cite[Proposition 7.3]{CJ23-T1}. Moreover, as pointed out earlier,
\cite[2.2 Theorem and 2.4 Corollary]{DP} seems to provide a quick proof of the assertion above, so that at least in principle, the results in this proposition should be deducible
from op. cit.  Nevertheless, it seems best to provide a proof of Proposition ~\ref{transf.real}, at least for $ \epsilon^*$: the proof for the other functors will be similar. 
First observe that there is a natural map $ \epsilon^* \RHom(\rmK, \rmL) \ra  \RHom(\epsilon^*(\rmK), \epsilon^*(\rmL))$ for any
two objects $\rmK, \rmL \in \Spt(\k_{\rm mot}, \cE)$. If one takes $\rmL = \cE$, $\RHom(\rmK, \rmL) $ will denote
$\rmD(\rmK)$. Similarly $\RHom ( \epsilon^*(\rmK),  \epsilon^*(\rmL))$ will then denote 
$\rmD( \epsilon^*(\rmK))$.
\vskip .2cm
Now the proof of the assertion for the pre-transfer follows from the commutativity of the following diagrams where the composition of maps in the top row
 (bottom row) is $ \epsilon^*(\tr'(\rmf_{}))$ ($\tr'(\epsilon ^*(\rmf))$, \res), with the smash products denoting their derived versions, and $\rmK = \cE \wedge \rmX_+$ as in the Proposition:
\vskip .2cm
\[\xymatrix{{ \epsilon^*(\cE)} \ar@<1ex>[d]^{id} \ar@<1ex>[r] & { \epsilon^*(\rmK \wedge _{\cE} \rmD\rmK)} \ar@<1ex>[r]^{\cong}& { \epsilon^*(\rmK) \wedge_{\epsilon^*(\cE)}  \epsilon^*(\rmD\rmK)} \ar@<1ex>[ld]\\
{ \epsilon^*(\cE)} \ar@<1ex>[r] & { \epsilon^*(\rmK) \wedge_{\epsilon^*(\cE)} \rmD(\epsilon^*(\rmK))} }\]
\vskip .2cm
\[\xymatrix{ { \epsilon^*(\rmK) \wedge_{\epsilon^*(\cE)}  \epsilon^*(\rmD\rmK)} \ar@<1ex>[d] \ar@<1ex>[rr]^{(id \wedge_{\epsilon^*(\cE)}  \epsilon^*(\rmf )\wedge_{\epsilon^*(\cE)}  \epsilon^*(\rmf)) \circ (id \wedge_{\epsilon^*(\cE)} \Delta) \circ \tau} &&{{ \epsilon^*(\rmD\rmK)} \wedge_{\epsilon^*(\cE)} { \epsilon^*(\rmK)} \wedge_{\epsilon^*(\cE)} { \epsilon^*(\rmK)}} \ar@<1ex>[r]^>>{\epsilon ^*(e) \wedge_{\epsilon^*(\cE)} id} \ar@<1ex>[d] & { \epsilon^*(\cE) \wedge  _{\epsilon^*(\cE)} \epsilon^*(K)} \ar@<1ex>[d]^{id}\\
{ \epsilon^*(\rmK) \wedge_{\epsilon^*(\cE)} \rmD( \epsilon^*(\rmK))} \ar@<1ex>[rr]^{(id \wedge _{\epsilon^*(\cE)} \epsilon^*(\rmf) \wedge _{\epsilon^*(\cE)} \epsilon^*( f)) \circ (id \wedge _{\epsilon^*(\cE)}\Delta) \circ \tau}  && {\rmD( \epsilon^*(\rmK)) \wedge_{\epsilon^*(\cE)}   \epsilon^*(\rmK) \wedge _{\epsilon^*(\cE)}  \epsilon^*(\rmK)} \ar@<1ex>[r]^>>{e \wedge_{\epsilon^*(\cE)} id} & { \epsilon^*(\cE) \wedge _{\epsilon^*(\cE)} \epsilon^*(\rmK)}}\]
\vskip .2cm \noindent
The isomorphism ${ \epsilon^*(\rmK \wedge _{\cE} \rmD\rmK)} {\overset {\cong} \ra}  { \epsilon^*(\rmK) \wedge_{\epsilon^*(\cE)}  \epsilon^*(\rmD\rmK)}$ in the top row may be obtained as follows.
First, using the fact that $\epsilon^*$ is pull-back from the big Nisnevich site of $\Speck$ to the corresponding big \'etale site,
one observes that $\epsilon^*$ commutes with smash products of pointed simplicial presheaves. Then, recalling the definition
of the smash product of spectra as a left Kan extension (see: \cite[Definition 4.2]{CJ23-T1}), one may see that $\epsilon^*$ commutes with smash products
and then smash products of module spectra over ring spectra. 

\end{proof}
\begin{corollary}
\label{transf.compat}
Assume that the group $\rmG$ is special and that $\rmf: \rmX \ra \rmX$ is a $\rmG$-equivariant map and let $\pi_{\rmY}: \rmE\times_{\rmG} (\rmY \times \rmX) \ra  \rmE\times_{\rmG} \rmY$ denote
 any one of the three cases considered  in Theorem ~\ref{thm.1}. Then $\epsilon ^*(tr(\rmf_{\rmY})) \simeq {\it tr}(\epsilon ^*(\rmf_{\rmY}))$,
where $tr(\rmf)$ denotes the transfer defined with respect to a motivic ring spectrum $\cE$ that is $\ell$-complete for a prime $\ell \ne char (k)$.
 \end{corollary}
\begin{proof} Proposition ~\ref{transf.real} proves the corresponding statement for the pre-transfer when $\rmY = Spec \, \k$. 
 Now the corresponding result holds for a general $\rmY$, since the corresponding pre-transfer $\tr'(\rmf_{\rmY}) = id_{\rmY_+} \wedge \tr'(\rmf)$. Next a detailed examination of the various steps
 in the construction of the transfer (see \cite[(8.3.3)]{CJ23-T1}  through \cite[(8.3.17)]{CJ23-T1}) show that they all pull-back to define the corresponding construction on the \'etale site. 
 (In fact, $\tr(\rmf_{\rmY})$ as in \cite[Definition 8.8]{CJ23-T1} is 
defined by first taking $id \wedge _G \tr'^{\rmG}(\rmf_{\rmY})$, where $\tr'^{\rmG}(\rmf_{\rmY})$ is the $\rmG$-equivariant pre-transfer.)
\end{proof}
\begin{remark}
 Several results on how to deduce splittings on generalized motivic cohomology theories, from splittings produced by the transfer
 at the level of \'etale cohomology are discussed in Propositions ~\ref{stable.splitting.2} and ~\ref{splitting.et}. These make use
 of the compatibility of the pre-transfer with \'etale realization as proven in Proposition ~\ref{transf.real}.
\end{remark}

\section{\bf Transfer and stable splittings in the Motivic Stable Homotopy category}
\vskip .2cm
%We will restrict to the transfers considered in Definition ~\ref{transfer:def}.
We will briefly recall the context and framework for the discussion in this section.
$\rmG$ is a linear algebraic group over a perfect field $k$,
$\rmp: \rmE \ra \rmB$ is a $\rmG$-torsor over  $k$, and $\pi_{\rmY}: \rmE \times_{\rmG}  (\rmX \times \rmY) \ra \rmE \times_{\rmG}  \rmY$ is the projection
 considered in the contexts (a) or (b) in Theorem ~\ref{thm.1}. (The splitting results for the context (c) considered in Theorem ~\ref{thm.1}(c) can be 
deduced from the case in Theorem ~\ref{thm.1}(b) in view of the compatibility of the transfer maps proved in Theorem ~\ref{thm.1}(i).)
 Recall that $\rmX$, $\rmY$ denote unpointed simplicial presheaves  (defined on $\Sm_{\k}$) provided  with  actions by $\group$ and so that $\Sigma_{\T}\rmX_+$ is 
dualizable in $\SH(\k)$ (or $\cE\wedge X_+$ is dualizable in $\SH(\k, \cE)$ when considering ring spectra $\cE$ 
other than the motivic sphere spectrum). $\rmf:X \ra X$ is a $\group$-equivariant map.
\vskip .2cm 
We next provide a quick review of the results on splitting obtained by making use of the transfer: these form some of the main results of the present paper.
In order to obtain splittings in the motivic stable homotopy category, there are essentially {\it two distinct techniques} we pursue here making use of the transfer as a stable map, each with its own advantages. Both of these apply to actions of all linear algebraic groups, {\it irrespective of whether they are  special}.
  %In fact, we provide {\it two distinct strategies} to establish such stable splittings, each with its own advantages. 
Both start with the observation that
 the multiplicative
 property of the transfer as in Theorem ~\ref{thm.1}(ii) and (iii)
 shows that in order to prove $\pi_{\rmY}^*$ is a split monomorphism, it suffices to show $\tr(\rmf_{\rmY})^*(1)= \tr(\rmf_{\rmY})^*\pi_{\rmY}^*(1)$ is a unit. Both approaches also make use of the
base-change property of the transfer as in Proposition ~\ref{base.change} and then reduce to checking this for 
simpler situations. They both apply to all linear algebraic groups, irrespective of whether they are special (that is,
in Grothendieck's classification as in \cite{Ch}), and in particular to 
all split orthogonal groups, which are known to be non-special. However, for the sake of keeping the  discussion simpler, we will restrict to groups that are special while 
considering the second approach.
\vskip .2cm \noindent
\subsection{Splittings via the Grothendieck-Witt ring of the base field $k$.} \index{Grothendieck-Witt ring}
\label{split.Groth.Witt.grp} \index{Grothendieck-Witt ring and splittings}
In this approach we assume that the class $\tau_{\rmX}^*(1)$ is a unit in the
Grothendieck-Witt ring  (or the Grothendieck-Witt ring with characteristic exponent of the base field
inverted) and use that to obtain splittings directly, first at the level of the pre-transfer. This method is rather limited to those schemes $\rmX$ for which it
is possible to compute $\tau_{\rmX}^*(1)$ in the Grothendieck-Witt ring. Such a computation is carried out  in
 \cite[Theorem 1.2 and Corollary 1.3]{JP23},  where $\rmX= \rmG/\N(T)$, for a connected split reductive group (or more generally a split linear algebraic group)
$\rmG$ over a perfect field and $\N(T)$ the normalizer of a split maximal torus in $\rmG$. Therefore, at present this technique 
 only applies to the above case. In the discussion below, we will only consider the case where $char(k) =0$. In positive characteristics $p$,
 the same discussion applies by replacing the sphere spectra everywhere by the corresponding sphere spectra with the prime $p$ inverted, or
completed away from $p$ as discussed in \cite[Definition 1.1]{CJ23-T1}.
\vskip .1cm
{\it Case 1}: Here we will assume the group $\rmG$ is {\it special}. Since the group $\rmG$ is assumed to be special, for each fixed integer $m \ge 1$,  the 
map $\rmp:\rmE \ra \rmB$ is a Zariski locally trivial principal $\rmG$-bundle and let $\tilde \rmp: {\widetilde \rmE} \ra {\widetilde \rmB}$ 
denote the induced map where $\widetilde \rmB$ is the affine replacement.
 Let $\{\rmU_i|i\}$ denote a Zariski open cover of $\widetilde \rmB$ over which the map $\tilde \rmp$ trivializes so that $\pi_{\rmY|\rmU_i} = \rmU_i \times (\rmY \times \rmX) \ra \rmU_i \times \rmY$. 
 \vskip .1cm
 Let $\tr: {\rm \Sigma}^{\infty}_{\T} ({\widetilde {\rmE}}\times_{\rmG} \rmY)_+ \ra {\Sigma^{\infty}_{\T}}  ({\widetilde {\rmE}}\times_{\rmG} (\rmY \times \rmX))_+$ denote the transfer defined in \cite[Definition 8.6]{CJ23-T1}. Then one may observe that $\tr_{|\rmU_i}: {\rm \Sigma}^{\infty}_{\T}(\rmU_i \times \rmY)_{+} \ra
 {\Sigma^{\infty} _{\T}}(\rmU_i \times \rmY)_{+} \wedge {\Sigma^{\infty}_{\T}}\rmX_+$ is just $id_{{\Sigma^{\infty}_{\T}}(\rmU_i \times \rmY)_{+}} \wedge 
 tr_{\rmX}'$, where $tr_{\rmX}'$ denotes the pre-transfer considered in  \cite[Definition 8.2(ii)]{CJ23-T1}. Therefore, if $\tau_{\rmX}^*(1)$ is a unit in
 the Grothendieck-Witt group (or in the Grothendieck-Witt group with the characteristic exponent of the base field inverted), (where $\tau_{\rmX}$ is the trace defined in \cite[Definition 8.2(iv)]{CJ23-T1}), then the composition, ${\Sigma^{\infty}_{\T}}\pi_{\rmY,+} \circ \tr_{\rmX}$, where $ \pi_{\rmY}$ is the projection
 ${\widetilde \rmE}\times_{\rmG}(\rmY \times \rmX) \ra {\widetilde \rmE}\times_{\rmG}\rmY$, will be homotopic to the identity
 over each $\rmU_i$. 
 \vskip .1cm
 Now let $\rmh^{*, \bullet}$ denote a generalized motivic cohomology theory defined 
 with respect to a motivic spectrum. Then using a Mayer-Vietoris argument and observing that each ${\widetilde \rmB}$ is quasi-compact,
  the splitting  over $\rmU_i$ of the map $\pi_{\rmY|\rmU_i}$ shows that the composite map  $tr^* \circ\pi_{\rmY}^*: \rmh^{*, \bullet}({\widetilde \rmE}{\underset {\rmG} \times }\rmY) {\overset {\pi_{\rmY}^*} \ra} \rmh^{*, \bullet}({\widetilde \rmE}{\underset {\rmG} \times }(\rmY \times \rmX)) {\overset {tr^*} \ra} \rmh^{*, \bullet}({\widetilde \rmE}{\underset {\rmG} \times }\rmY)$ is an isomorphism. 
 When we vary $\rmB $ over finite dimensional approximations $\{\BG^{\it gm,m}|m\}$, the same therefore
 holds on taking the colimit of the $\BG^{\it gm,m}$ over $m$ as $m \ra \infty$, as we have shown the transfer maps
 are compatible as $m$ varies: see Theorem ~\ref{thm.1}(i). (Here the colimit of the
 $\BG^{\it gm,m}$ will pullout of the generalized cohomology spectrum as a homotopy inverse limit, and then one uses the usual $lim^1$-exact sequence to draw the desired conclusion.)
 \vskip .1cm
 {\it Case 2}: Here we will let $\rmG$ denote {\it any linear algebraic group}. We first recall from \cite[Definition 8.2(iii)]{CJ23-T1},
 that the $\rmG$-equivariant pre-transfer is given by 
 $id_{\rmY} \times \tr'^{\rmG}(id): \rmY \times \mbS^{\rmG}_{\k} \ra \rmY  \times (\mbS^{\rmG}_{\k} \wedge \rmX_+)$, with $\tr'^{\rmG}(id): \mbS^{\rmG}_{\k} \ra \mbS^{\rmG}_{\k} \wedge \rmX_+$ 
 the $\rmG$-equivariant pre-transfer in \cite[Definition 8.2(ii)]{CJ23-T1}. Moreover, the composition of the above
 pre-transfer and the projection $\rmY \times (\mbS^{\rmG}_{\k} \wedge \rmX_+) \ra \rmY \times \mbS^{\rmG}_{\k}$ is $id_{\rmY} \times \tau_{\rmX}$,
 {\it where we view $\mbS^{\rmG}_{\k}$ as a spectrum in} $\widetilde \Spt^{\rmG}(\k_{\rm mot})$. 
 In view of the assumption that $\tau_{\rmX}^*(1)$ is a unit in the Grothendieck-Witt ring of $\k$, this composite map is
 a weak-equivalence mapping $\rmY \times \mbS^{\rmG}_{\k} $ to itself, again viewing $\mbS^{\rmG}_{\k}$ as a spectrum in $\widetilde \Spt^{\rmG}(\k_{\rm mot})$. Next 
 consider the composite map
 \[{\widetilde \rmE}\times_{\rmG}^{et}({\rm a}\epsilon^*(\rmY \times \mbS^{\rmG}_{\k})) {\overset {{\widetilde \rmE}\times_{\rmG}^{et}(a\epsilon^*\tr'^{\rmG}(id_{\rmY+}))} \longrightarrow} {\widetilde \rmE}\times_{\rmG}^{et}(a\epsilon^*(\rmY \times (\mbS^{\rmG}_{\k} \wedge \rmX_+))) {\overset {{\widetilde \rmE}\times_{\rmG}^{et}(a\epsilon^*(pr))} \longrightarrow} {\widetilde \rmE}\times_{\rmG}^{et}(a\epsilon^*(\rmY \times \mbS^{\rmG}_{\k}))\]
 which lives over the small \'etale site of ${\widetilde \rmE} \times_{\rmG}^{et} \rmY$.
 Next we smash each term of the above spectra indexed by ${\rmT}_{\rmV}$ with   the Thom-space of
 the complimentary bundle $\epsilon^*(\eta^{\rmV})$ over the base $\epsilon^*(\widetilde \rmE_{\rmY})$. Working locally on the small \'etale site of 
 ${\widetilde \rmE} \times_{\rmG}^{et} \rmY_+$, one can see that the resulting composite map of spectra identifies with ${\widetilde \rmE} \times_{\rmG}^{et}(id_{\rmY} \times \epsilon^*(\tau_{\rmX}))$.
 This is a weak-equivalence since $\tau_{\rmX}:\mbS^{\rmG}_{\k} \ra \mbS^{\rmG}_{\k}$ is a weak-equivalence, when $\mbS^{\rmG}_{\k}$ is viewed as a spectrum 
 in ${\widetilde \Spt}^{\rmG}(\k_{\rm mot})$. Therefore, it remains a weak-equivalence of spectra on applying
 $\rmR\epsilon_*$ and collapsing the section from $\tilde \rmE_{\rmY}$
  as in \cite[Steps 3-5 of 8.3:Construction of the transfer]{CJ23-T1}. In fact, now the corresponding map is
  \be \begin{equation}
  \rmR\epsilon_*(id_{({\widetilde \rmE} \times_{\rmG}^{et} {\rm a}\epsilon^*(\rmY))_+} \wedge \epsilon^*(\tau_{\rmX})): \rmR\epsilon_*(\epsilon^*\mbS_{\k}) \wedge \rmR\epsilon_*(({\widetilde \rmE} \times_{\rmG}^{et} {\rm a}\epsilon^*(\rmY))_+ \ra \rmR\epsilon_*(\epsilon^*\mbS_{\k}) \wedge \rmR\epsilon_*(({\widetilde \rmE} \times_{\rmG}^{et} \rmY)_+ .
  \end{equation} \ee
 %Therefore, we take the derived smash product of the above map with $\mbS_{\k}$ over $\rmR\epsilon_*(\epsilon^*({\mbS}_{\k}))$ to get an induced map: $\mbS_{\k} \wedge \rmR\epsilon_*(({\widetilde \rmE} \times_{\rmG}^{et} {\rm a}\epsilon^*(\rmY))_+  \ra \mbS_{\k} \wedge \rmR\epsilon_*(({\widetilde \rmE} \times_{\rmG}^{et} {\rm a}\epsilon^*(\rmY))_+ $
 %which is a weak-equivalence.
 Finally we apply a generalized motivic cohomology theory $\rmh^{*, \bullet}$ to the above maps to observe that the composite map
 \[\rmh^{*, \bullet}(\rmR\epsilon_*(\epsilon^*{\mbS_{\k}}) \wedge \rmR\epsilon_*({\widetilde \rmE}\times_{\rmG}^{et}({\rm a}\epsilon^*(\rmY)))_+) {\overset {\pi_{\rmY}^*} \longrightarrow} \rmh^{*, \bullet}(\rmR\epsilon_*(\epsilon^*{\mbS_{\k}}) \wedge R\epsilon_*({\widetilde \rmE}\times_{\rmG}^{et}(a\epsilon^*(\rmY \times \rmX))))_+)\]
 \[{\overset {\tr(id_{\rmY}^*)} \longrightarrow} \rmh^{*, \bullet}(\rmR\epsilon_*(\epsilon^*{\mbS_{\k}}) \wedge \rmR\epsilon_*({\widetilde \rmE}\times_{\rmG}^{et}({\rm a}\epsilon^*(\rmY)))_+)  \]
is an isomorphism.
\vskip .2cm
The {\it main advantage} of this method is that it provides splittings for all generalized motivic cohomology theories, whenever
the above computation of the trace $\tau_{\rmX}^*(1)$ can be carried out in the Grothendieck-Witt ring of the base field, 
{ independent of whether the group $\rmG$ is special}. One may see a detailed discussion of splittings obtained this way in Theorems 
 ~\ref{stable.splitting.1} and  ~\ref{split.0} discussed below as well as Corollary ~\ref{transf.appl.1}.

\vskip .2cm

 \begin{theorem} \index{Grothendieck-Witt ring and splittings}
\label{stable.splitting.1}
Let $\pi_{\rmY}: \rmE\times_{\rmG}(\rmY \times \rmX) \ra \rmE\times_{\rmG} \rmY$ denote a map as in one of the three cases considered in Theorem ~\ref{thm.1}. Assume $\rmG$ is a
 split linear algebraic group. 
In case $\rmG$ is not special, we will also assume the field $\k$ is infinite and 
 we will also assume the field $k$ satisfies the hypothesis ~\eqref{etale.finiteness.hyps}. Let $\rmM$ denote a motivic spectrum.
\vskip .1cm
 Then the map induced by $\tr(id_{\rmY})^*$ provides a splitting to the map 
 \[\pi_{\rmY}^*: \rmh^{*, \bullet}(\rmE\times_{\rmG}\rmY, \rmM) \ra  \rmh^{*, \bullet}(\rmE\times_{\rmG}(\rmY \times \rmX), \rmM)\]
in the following cases:
\begin{enumerate}[\rm(i)]
 \item ${\Sigma^{\infty}_{\T}}\rmX_+$ is 
dualizable in the motivic homotopy category $\SH(\k)$, the trace $\tau_{X}^*(1)$ is a unit in the Grothendieck-Witt ring of the base 
field $\k$ and $\rmM$ denotes any motivic spectrum. In particular, this holds if $\rmX$ and $\rmY$ are smooth schemes of finite
type over $\k$ and $char (k) =0$, provided $\tau_{\rmX}^*(1)$ is a unit in the Grothendieck-Witt ring of $\k$.

\item $Char (k)=p>0$, $\cE$ denotes any one of the ring spectra, (a) $\mbS_{\k}[p^{-1}]$, (b) $ \mbS_{\k,(\ell)}$ or (c) $\cE= {\widehat \mbS}_{\k, \ell}$ for some prime $\ell \ne p$ and 
$\cE \wedge X_+$
is dualizable in $\SH(\k, \cE)$, $\rmM \in \Spt(\k_{\rm mot}, \cE)$. We will further assume 
 that 
the corresponding trace $\tau_{\rmX}: \cE \ra \cE$ is a unit in the corresponding variant of the Grothendieck-Witt ring, that is,
$[\cE, \cE]$, which denotes stable homotopy classes of maps from $\cE$ to $\cE$. In particular, this holds if $\rmX$ and $\rmY$ are smooth schemes of finite
type over $\k$, provided $\tau_{\rmX}^*(1)$ is a unit in the above variant of Grothendieck-Witt ring of $\k$.
\end{enumerate}
\end{theorem}
 \vskip .1cm
 The main example, where one is able to compute the motivic trace in the Grothendieck-Witt group is for $\rmG/\N(T)$, where $\rmG$ is a split linear algebraic group and $\N(T)$ denotes the normalizer of 
a split maximal torus in $\rmG$: see \cite{Lev18} for partial results in
this direction and see \cite[Theorem 1.2 and Corollary 1.3]{JP23}  for a computation in the general case. This yields the following result.
\vskip .1cm
 \begin{theorem} (See \cite[Theorem 1.2 and Corollary 1.3]{JP23} and also \cite[Proposition 2.2]{JP22}.) \index{Grothendieck-Witt ring and splittings}
	\label{split.0}
	Let $\rmG$ denote a {\it split}  linear algebraic group over the given base field $k$: we will assume $k$ is infinite when $\rmG$ is not special. Let $\rmh^{*, \bullet}$ denote a generalized motivic cohomology theory defined
	with respect to a motivic spectrum (with $p$ inverted, if $char({\it k}) =p>0$.)  Then,  
	with $\N(T)$ denoting the normalizer of a split maximal torus in $\rmG$, 
	and with $\rmp: {\Sigma^{\infty}_{\T}}\BN(T)_+ \ra {\Sigma^{\infty}_{\T}}\BG_+$ denoting the map induced by the inclusion $\N(T) \ra {\rm G}$, the induced map
	\[\rmp^*:{\rmh^{*, \bullet}(\BG_+)} \ra {\rmh^{*, \bullet}(\BN(T)_+)}\]
	is {\it split injective}. In particular, when the group $\rmG$ is special, the above splitting holds for Algebraic K-theory (integral in characteristic $0$ and with finite coefficients prime to
	the characteristic, in general.)
\end{theorem}
\vskip .1cm
\begin{corollary}
 	\label{transf.appl.1}
 	Assume in addition to the the hypotheses of Theorem ~\ref{split.0}, that the following hypotheses hold for (i) through (iii).  Let $\rmM$ denote any motivic spectrum if the base field is of characteristic $0$ and let $\rmM$ denote
 	a motivic spectrum in $\Spt(\k_{\rm mot}, \mbS[{\it p}^{\rm -1}])$ if the base field is of characteristic ${\it p} >0$. Let
 	$\rmp: \rmE \ra \rmB$ denote the map appearing in one of the three cases (a) through (c) considered in Theorem ~\ref{thm.1}. 
 	\begin{enumerate}[\rm(i)]
     \item
 	 Let $\pi: \rmE{\underset {\group} \times}(\rmG/\NT) \ra \rmB$ denote the map induced
        by the projection $\rmG/\rmN(\rmT) \ra Spec \, \k$. Then the corresponding induced map
 	 \[\pi^*:   \rmh^{\bullet, *}(\rmB, M) \ra   \rmh^{\bullet, *}(\rmE{\underset {\rmG} \times}(\rmG/\NT), M)\]
 	 is a split monomorphism, where $\rmh^{*, \bullet}(\quad, M)$ denotes the generalized motivic cohomology theory defined with respect to the spectrum  $\rmM$.
     \item  Let $\rmY$ denote a $\rmG$-scheme or an unpointed simplicial presheaf provided with a $\rmG$-action. Let  
 	 $\rmq: \rmE{\underset {\rmG} \times}({\rmG}{\underset {\NT} \times}\rmY ) \ra \rmE{\underset {\rmG} \times} \rmY$ denote the map induced by the map ${\rmG}{\underset {\NT} \times}\rmY \ra \rmY$
 	  sending $(g, y) \mapsto gy$. Then, the induced map
 	 \[\rmq^*: \rmh^{\bullet, *}(\rmE{\underset {\rmG} \times} \rmY, M) \ra   \rmh^{\bullet, *}(\rmE{\underset {\rmG} \times}({\rmG}{\underset {\NT} \times}\rmY ), M)\]
 	 is also a split injection.
     \item Assume the base field $\k$ satisfies the finiteness hypotheses in ~\eqref{etale.finiteness.hyps}.
     Let $\cE$ denote a commutative ring spectrum in ${\Spt}(\k_{\rm mot})$, whose presheaves of homotopy groups are  all $\ell$-primary torsion for a
 fixed prime $\ell \ne char (k)$, and
let $\epsilon^*(\cE)$ denote the corresponding spectrum in $\Spt(\k_{et})$. 
We will further assume that $\rmM$ is a module spectrum over $\cE$. 
 Then the results corresponding to (i) through (iii) hold for $\rmh^*( \quad, \epsilon^*(\rmM))$ which is the generalized \'etale cohomology with respect to the \'etale spectrum $\epsilon^*(\rmM)$. 
       \end{enumerate}
 \end{corollary}
\vskip .2cm \noindent
{\bf Proofs of Theorems ~\ref{stable.splitting.1}, ~\ref{split.0} and Corollary ~\ref{transf.appl.1}}.
Clearly the statements in Theorem ~\ref{stable.splitting.1}(i) 
and Theorem ~\ref{split.0} follow readily 
in view of the above discussion in \ref{split.Groth.Witt.grp}. The proof of the statement in Theorem ~\ref{stable.splitting.1}(ii) is entirely similar with 
the motivic sphere spectrum $\mbS_{\k}$ replaced by the given motivic ring spectrum $\cE$.\qed
\vskip .1cm
In view of the results on 
motivic Euler characteristic of $\rmG/\NT$ proved in \cite[Theorem 1.2 and Corollary 1.3]{JP23}, 
Theorem  ~\ref{split.0} and Corollary ~\ref{transf.appl.1}
 follow from the above discussion. \qed
 \vskip .2cm \noindent
\begin{remark}
Special cases of the above theorem and Corollary, such as for Algebraic K-theory are particularly interesting.
Theorem ~\ref{split.0}, for Algebraic K-Theory in fact enables one to restrict the structure group from ${\group}$ to $\N(T)$
(and then to $\rm\rmT$ by ad-hoc arguments) in several situations. Taking ${\group}=\GL_n$, this becomes a
{\it splitting principle} reducing 
problems on vector bundles to corresponding problems on line bundles. 
 \end{remark}
\vskip .1cm
The main {\it disadvantages} of this method are as follows.
Computing the trace associated to a scheme in the Grothendieck-Witt ring  is
extremely difficult for many schemes, and possibly not do-able with the present technology. As pointed out above, the only case where this seems do-able at 
present is for $\rmG/\N(T)$, where $\rmG$ is a split linear algebraic group and $\N(T)$ denotes the normalizer of 
a split maximal torus in $\rmG$. Moreover, the above discussion is only for the
 case the self-map $\rmf: \rmX \ra \rmX$, (with $\rmX=\rmG/\N(T)$) is the identity: it is far from clear how to carry out a similar computation
 in the Grothendieck-Witt ring for a general self-map $\rmf$, even for the same $\rmX$.
\vskip .3cm \noindent
\subsection{Splittings for slice-completed generalized motivic cohomology theories}
 First we define slice-completed  generalized motivic cohomology theories.
 \begin{definition}
 \label{completed.gen.coh} \index{slice completed cohomology theory}
  For a smooth scheme $\rmY$ (smooth ind-scheme $\Y=\{\rmY_m|m\}$), we define {\it the slice completed
 generalized motivic cohomology spectrum}  with respect to a motivic spectrum $\rmM$ to be 
 \[\hat \rmh(\rmY, \rmM) = \holimn \H_{Nis}(\rmY, s_{\le n}\rmM) \simeq \H_{Nis}(\rmY, \holimn s_{\le n}\rmM)\]
 \[(\hat \rmh(\Y, \rmM) = \holimm \holimn \H_{Nis}(Y_m, s_{\le n} \rmM) \simeq \holimm \H_{Nis}(\rmY_m, \holimn s_{\le n}\rmM)),\]
 where $s_{\le n} \rmM $ is the homotopy cofiber of the map 
$f_{n+1}\rmM \ra \rmM$ and $\{f_n \rmM|n\}$ is the {\it slice tower} for $\rmM$ with $f_{n+1}\rmM$ being the $n+1$-th connective cover of $\rmM$. ($ {\mathbb H}_{\rm Nis}(\rmY, \rmF)$ and $ {\mathbb H}_{\rm Nis}(\rmY_m, \rmF)$ denote the generalized hypercohomology spectrum
with respect to a motivic spectrum $\rmF$ computed on the Nisnevich site.) The corresponding homotopy groups for maps from ${\Sigma^{\infty}_{\T}}(\rmS^{\it u} \wedge {\mathbb G}_{\it m}^{\it v})$ to the above spectra will be denoted $\hat \rmh^{u+v, v}(\Y, \rmM)$.
One may define the completed generalized \'etale cohomology spectrum of a scheme with respect to an $\rmS^1$-spectrum
 by using the Postnikov tower in the place of the slice tower in a similar manner.
 \end{definition}
 \begin{remark}
 \label{slice.compl.sp} By \cite[18.1.8]{Hirsch}, the homotopy inverse limit $\holimn s_{\le n}\rmM$ belongs to $\Spt(\k)$. We may, therefore,
  define the {\it slice completion} of the spectrum $\rmM$ to be $\holim s_{\le n}\rmM$ (denoted henceforth by $\hat \rmM$) and define $\rmM$ to be slice-complete, if
  the natural map $\rmM \ra \hat \rmM$ is a weak-equivalence.  Therefore,
  one may see that $\hat {\it h}(\rmY, \rmM) = {\it h}(\rmY, \hat \rmM)$ and $\hat {\it h}(\Y, \rmM) = {\it h}(\Y, \hat \rmM)$. Several important spectra, like the
  spectrum representing algebraic K-theory and algebraic cobordism, are known to be slice-complete.
\end{remark}
 \vskip .2cm
 This method makes strong use of 
 the fact that the transfer map we have constructed is a map in the appropriate stable homotopy category and, therefore
 induces a map of the corresponding motivic (or \'etale)  Atiyah-Hirzebruch spectral sequences: therefore, it suffices to show the transfer
 induces a splitting at the $\rmE_2$-terms of this spectral sequence. The spectral sequence can be shown to
 convergence strongly once we replace a given spectrum $\rmM$ with one of its slices, as shown in the discussion in 
 the proof of Theorem ~\ref{stable.splitting.2}. %    section ~\ref{slice.compl.splittings}.
 \vskip .1cm
The multiplicative properties of the slice filtration
that a natural pairing of the slices of a motivic spectrum lift to the category $\Spt(\k_{\rm mot})$ from
the corresponding motivic stable homotopy category
 were verified in \cite{Pel}. Therefore, it follows that the  
motivic spectra that define the
$\rmE_2$-terms of the motivic Atiyah-Hirzebruch spectral sequence, that is the slices of the given motivic spectrum, are modules over the motivic 
Eilenberg-Maclane spectrum $\H({\mathbb Z})$. Then the multiplicative property of the transfer as in Proposition ~\ref{multiplicative.prop.0} and
 Corollary ~\ref{multiplicative.prop.1} reduce to checking that we obtain a splitting for motivic cohomology.
 \vskip .2cm
 Next, we make use of the
base-change property of the transfer as in Proposition ~\ref{base.change} and then reduce to checking this for the action of trivial groups, that is for the pre-transfer.
See, for example,  Proposition ~\ref{trace.2}. 
\vskip .1cm
At this
point it is often very convenient, as well as necessary, { to know that the transfer is compatible with passage to simpler situations, for example, to a change of the base field to
one that is separably or algebraically closed and with suitable realizations, that is either the \'etale realization or the Betti realization}. 
A {\it main advantage} of this approach is that it would be only necessary to compute $\tr(\rmf)^*(1)$ and the trace $\tau_{\rmX}^*(1)$
after such reductions and realizations, which are readily do-able for a large number of schemes $\rmX$: see Proposition ~\ref{transf.real} 
 and Corollary ~\ref{transf.compat}.
{\it Another advantage} is that it addresses affirmatively the important question if the pre-transfer and transfer are compatible
with such reductions and realizations. Moreover, by this method, one can allow any self-map $\rmf:\rmX \ra \rmX$ and compute
 the corresponding trace $\tau_{\rmX}(\rmf)$. One can consult Theorem ~\ref{stable.splitting.2} and Corollary  ~\ref{Kth.splitting} discussed below, for examples of splittings obtained this way. 
 Though a variant of the following theorem holds for groups that are non-special, with certain modifications,  we will restrict to groups 
 that are special, mainly to keep the discussion simpler.
\vskip .2cm

 \begin{theorem}
\label{stable.splitting.2} \index{slice completed cohomology theory and splittings}
Assume in addition to the hypotheses of Theorem ~\ref{stable.splitting.1} that the following hold. Assume the group $\rmG$ is special and 
 let $\rmf:\rmX \ra \rmX$ denote a $\group$-equivariant map and let $\rmf_Y=id_{Y_+} \wedge \rmf: \rmY_+ \wedge \rmX_+ \ra \rmY_+ \wedge \rmX_+$. The map induced by $\tr(\rmf_{\rmY})^*$ provides a splitting to the map $\pi_{\rmY}^*: \hat \rmh^{*, \bullet}(\rmE\times_{\rmG}\rmY, \rmM) \ra \hat \rmh^{*, \bullet}(\rmE\times_{\rmG}(\rmY \times \rmX), \rmM)$
in the following cases:
\vskip .1cm
\begin{enumerate}[\rm(i)]
 \item ${\Sigma^{\infty}_{\T}}\rmX_+$ is 
dualizable in $\SH(\k)$  and $\tr(\rmf_{\rmY})^*(1)$ is 
a unit in $\rmH^{0, 0}(\rmE\times_{\rmG} \rmY, {\mathbb Z}) \cong \CH^0(\rmE\times_{\rmG} \rmY)$. In particular, this holds if $\rmX$ and $\rmY$ are smooth schemes of finite
type over $\k$ and $char (k) =0$, provided $\tr(\rmf_{\rmY})^*(1)$ is 
a unit in $\rmH^{0, 0}(\rmE\times_{\rmG} \rmY, {\mathbb Z})$.
\vskip .2cm
\item A corresponding result
also holds for the following alternate scenario:
\vskip .2cm
(a) The field $k$ is of positive characteristic $p$, $\mbS_{\k}[{\rm p}^{\rm -1}] \wedge \rmX_+$ is dualizable in $\SH(\k, \mbS_{\k}[{\rm p}^{\rm -1}])$, $\rmM \in \Spt(\k_{\rm mot}, \mbS_{\k}[{\rm p}^{\rm -1}])$ and
$\tr(\rmf_{\rmY})^*(1)$ is a unit in $\rmH^{0, 0}(\rmE\times_{\rmG} \rmY, {\mathbb Z}[{\it p}^{-1}]) \cong \CH^0(\rmE\times_{\rmG} \rmY, Z[{\rm p}^{\rm -1}])$.
\vskip .2cm
(b) The field $k$ is of positive characteristic $p$, $\cE= \mbS_{\k,(\ell)}$ (or $\cE= {\widehat \mbS}_{\k, \ell}$) for some prime $\ell \ne p$ and $\cE \wedge X_+$
is dualizable in $\SH(\k, \cE)$, $\rmM \in \Spt(\k_{\rm mot}, \cE)$ and $\tr(\rmf_{\rmY})^*(1)$ is a unit in $\rmH^{0, 0}(\rmE\times_{\rmG} \rmY, {\mathbb Z}_{(\ell)}) \cong \CH^0(\rmE\times_{\rmG} \rmY, Z_{(\ell)})$
 ($\H^{0,0}(\rmE\times_{\rmG} \rmY, {\mathbb Z}\compl_{\ell}) \cong \CH^0(\rmE\times_{\rmG} \rmY, Z\compl_{\ell})$, \res).
 (Here $\mbS_{\k,(\ell)}$ (${\widehat \mbS}_{\k, \ell}$) denotes the localization of the motivic spectrum ${\mbS_{\k}}$ at the prime ideal $(\ell)$
(the completion at $\ell$, \res).) 
\item
Assume the following: (a) both $\rmX$ and $\rmY$ are smooth 
schemes of finite type over $\k$ provided with $\rmG$-actions, and (b)  both $\rmY$ and $\rmE\times_{\rmG}\rmY$ are connected.
Let $\rmH^{*, \bullet}(\quad, \rmR)$ denote motivic cohomology with coefficients in the commutative Noetherian ring $\rmR$ with a unit.
Then under the isomorphism 
\[\rmH^{0,0}(\rmE\times_{\rmG} \rmY, \rmR) {\overset {\cong } \ra} \rmH^{0,0}( \rmY, \rmR)\]
$\tr(\rmf_{\rmY})^*(1)$ identifies with $(id_{\rmY} \times \tau_{\rmX}(\rmf))^*(1)$. Therefore, the former is a unit if and  only if
the latter is. Under the assumption that $\tr(\rmf_{\rmY})^*(1)$ is a unit, then $\tr(\rmf_{\rmY})^*$ provides a splitting 
to the map $\pi_{\rmY}^*: \hat \rmh^{*, \bullet}(\rmE\times_{\rmG}\rmY, \rmM) \ra \hat \rmh^{*, \bullet}(\rmE\times_{\rmG}(\rmY \times \rmX), \rmM)$,
provided the slices of the motivic spectrum $\rmM$ are modules over the motivic Eilenberg-Maclane spectrum $\rmH(\rmR)$.
\end{enumerate}
\end{theorem}
\begin{proof} %{\bf Proof of Theorem ~\ref{stable.splitting.2}}. 
\label{slice.compl.splittings}
%\vskip .1cm \index{slice completed cohomology theory and splittings}
 One key observation here is that  the map $\tr(\rmf_{\rmY})$ being a stable map, it induces a map of the stable slice spectral sequences for 
$\rmh^{*, \bullet}(\rmE\times_{\rmG}(\rmX \times \rmY), M)$ and  $\rmh^{*, \bullet}(\rmE\times_{\rmG}\rmY , M)$. (One may observe that 
the  slice spectral sequences converge only conditionally, in general: the convergence issues will be discussed below.)
Next we will show, under the hypotheses of the theorem, that multiplication by 
$\tr(\rmf_{\rmY})^*(\pi_{\rmY}^*(1))$ induces a splitting of the corresponding $\rmE_2$-terms of the above
spectral sequences. For this, recall that the multiplicative properties of the slice filtration, verified in \cite{Pel}, shows that these 
$\rmE_2$-terms are  modules over the motivic cohomology and that, in fact these $\rmE_2$-terms are 
defined by motivic spectra (that is, the slices) that are module
spectra over the motivic Eilenberg-Maclane spectrum.
\vskip .1cm 
Therefore, under these assumptions, the multiplicative property of the transfer as in Corollary ~\ref{multiplicative.prop.1} with $\rmA$ there denoting the motivic Eilenberg-Maclane
spectrum ${\mathbb H}({\mathbb Z})$ and the module spectrum $\rmM$ there denoting the module spectra defining the above $\rmE_2$-terms, shows
that $\tr(\rmf_{\rmY})^*\circ \pi_{\rmY}^*$ induces a map of the $\rmE_2$-terms of the above motivic Atiyah-Hirzebruch spectral sequences. 
That is, we reduce to proving $\tr(\rmf_{\rmY})^*\circ \pi_{\rmY}^*$ induces an isomorphism on the motivic cohomology of $\rmE\times_{\rmG}\rmY$,
modulo the convergence issues of the spectral sequence.
\vskip .1cm
Next we discuss convergence issues of the spectral sequence. Since the map $\tr(\rmf_{\rmY})^*\circ \pi_{\rmY}^*$ induces an isomorphism at the $\rmE_2$-terms, and therefore at all the $\rmE_r$-terms for $r \ge 2$,
it follows that it induces an isomorphism of the inverse systems $\{\rmE_r|r \}$ and therefore an isomorphism of the $\rmE_{\infty}$-terms and the 
 derived $\rmE_{\infty}$-terms. (See \cite[(5.1)]{Board} for a description of the derived $\rmE_{\infty}$-terms. It is shown in 
\cite[Chapter XV, section 2]{CE}  that both the $\rmE_{\infty}$-terms and the derived $\rmE_{\infty}$-terms
are determined by the sequence $\rmE_r$, $r \ge 2$.)
\vskip .2cm
Next, let $\{\rmY_m|m\}$ denote either one of the following ind-schemes: $\rmY_m = \rmE\times _{\rmG}\rmY$, for all $m\ge 1$ or 
$\rmY_m=\EG^{\it gm,m}\times_{\rmG}\rmY$, $m \ge 1$.
The next observation is that for every fixed integer $n$ and $m$, on replacing the spectrum $\rmM$ by $s_{\le n}\rmM$, the corresponding
slice spectral sequence for the schemes $\rmY_m$ converge strongly: this is clear since the $\rmE_1^{\it u,v}$-terms will vanish for all $u>n$ and also for $u<0$.
(See \cite[Theorem 7.1]{Board}.)
That $\rmE_1^{\it u, v}=0$ for $u<0$ or $u>n$ follows from the identification of the $\rmE_1$-terms of the spectral sequence in terms of the slices of the $\rmS^1$-spectrum
forming the $0$-th term in the associated $\Omega$-${\mathbb P}^1$-spectrum: see \cite[Proof of Theorem  11.3.3]{Lev}. Moreover, the
abutment of the spectral sequence are the homotopy groups of  the slice-completion of the $\rmS^1$-spectrum forming the corresponding $0$-th term.
Then, it follows therefore that for each fixed integer $m$ and $n$, the composite map 
\[\tr(\rmf_{\rmY})^* \circ \pi_{\rmY}^* :{\mathbb H}_{Nis}(Y_m, s_{\le n} M) \ra {\mathbb H}_{Nis}(Y_m, s_{\le n} M)\]
is a weak-equivalence, provided $\tr(\rmf_{\rmY})^*\circ \pi_{\rmY}^*$ induces an isomorphism on the motivic cohomology of $\rmE\times_{\rmG}\rmY$,
where ${\mathbb H}_{Nis}$ denotes the hypercohomology spectrum on the Nisnevich site. 
\vskip .1cm
Therefore, from the compatibility of the transfer as $m$ varies as proven in \cite[8.3: Step 2]{CJ23-T1},  it 
follows that the composite map
\[\tr(\rmf_{\rmY})^* \circ \pi_{\rmY}^*:\holim_{\infty \leftarrow m}\holim_{\infty \leftarrow n} {\mathbb H}_{Nis}(Y_m, s_{\le n} M) \ra \holim_{\infty \leftarrow m}\holim_{\infty \leftarrow n} {\mathbb H}_{Nis}(Y_m, s_{\le n} M)\]
is a weak-equivalence. 
\vskip .2cm
Next observe that the above discussion already proves the
first statement in Theorem ~\ref{stable.splitting.2}(i).  Then the second statement there follows by observing that ${\Sigma^{\infty}_{\T}}\rmX_+$ is dualizable in $\Spt(\k_{\rm mot})$.
 These complete the proof of Theorem  ~\ref{stable.splitting.2}(i).
\vskip .1cm
In order to prove the variant in {\it Theorem ~\ref{stable.splitting.2}(ii)(a)}, it suffices to observe that the slices of the module spectrum $\rmM$ are now module spectra over
$\mbS_{\k}[{\it p}^{-1}]$ and that the zero-th slice of $\mbS_{\k}[{\it p}^{-1}] = \H({\mathbb Z}[{\it p}^{-1}])$. Then  essentially the same arguments as above apply, along with \cite[Theorem 7.1]{CJ23-T1} to complete the proof of statement (a). In order to prove the
 variant in Theorem ~\ref{stable.splitting.2}(ii)(b), it suffices to observe that the slices of the module spectrum $\rmM$ are module spectra over ${\mbS_{\k (\ell)}}$ (${\widehat \mbS}_{\k, \ell}$)
and that the zero-th slice of $\mbS_{\k,(\ell)}$ is $\H({\mathbb Z}_{(\ell)})$ (the zero-th slice of ${\widehat \mbS}_{\k, \ell}$ is $\H({\mathbb Z}\compl_{\ell})$, \res). 
In fact, both these statements follow readily by identifying the slice tower with the coniveau tower as in \cite{Lev}.
These complete the proof of Theorem ~\ref{stable.splitting.2}(ii)(a) and (b). 
%\vskip .1cm
%The \'etale variant in Theorem ~\ref{stable.splitting.2}(iii) follows first by observing that Postnikov sections of the module spectrum $\rmM$ are module spectra
% over the $\ell$-completed $\rmS^1$-sphere spectrum $\widehat {\Sigma^{\infty}}_{S^1, \ell}$. 
\vskip .2cm
We next discuss the  statement  Theorem ~\ref{stable.splitting.2}(iii).
 Observe that $\rmY$ and $\rmX$ are assumed
 to be smooth schemes of finite type over $\k$ and that the group $\rmG$ is assumed to be special. Therefore, we invoke Propositions
~\ref{trace.2}, ~\ref{no.need.special}. Proposition ~\ref{no.need.special} shows that the hypotheses of Proposition ~\ref{trace.2} are 
~satisfied by the motivic cohomology with respect to a commutative ring $\rmR$, so that Proposition ~\ref{trace.2} proves
the statement in Theorem ~\ref{stable.splitting.2}(iii).

\end{proof}

\begin{proposition}
\label{splitting.et}
 Assume in addition to the hypotheses of Theorem ~\ref{stable.splitting.2}(iii) that the scheme $\rmY$ is geometrically connected
 and that the base field $k$ satisfies the hypothesis 
~\eqref{etale.finiteness.hyps}, where $\ell$ is prime to the characteristic of $k$ and $\nu$ is a positive integer.
 \begin{enumerate}[\rm(i)]
 \item Then denoting by $\rmY_{\bar \k}$ the base change of $\rmY$ to the algebraic closure $\bar \k$, the class
 $(id_{\rmY} \times \tau_{\rmX}(\rmf))^*(1)$ identifies with the class $(id_{\rmY_{\bar \k}} \times \tau_{\rmX}(\rmf))^*(1) \in \rmH^0_{et}(\rmY_{\bar {\k}}, {\mathbb Z}/\ell^{\nu})$
 under the isomorphisms 
 \[\rmH^{0,0}(\rmY, {\mathbb Z}/\ell^{\nu}) {\overset {\cong} \ra} \rmH^{0, 0}(\rmY_{\bar {\k}}, {\mathbb Z}/\ell^{\nu}) {\overset {\cong } \ra} \rmH^0_{et}(\rmY_{\bar {\k}}, {\mathbb Z}/\ell^{\nu}).\]
 Therefore, under the  hypotheses that $(id_{\rmY_{\bar \k}} \times \tau_{\rmX}(\rmf))^*(1) \in \rmH^0_{et}(\rmY_{\bar {\k}}, {\mathbb Z}/\ell^{\nu})$ is a unit,
 $\tr(\rmf_{\rmY})^*$ provides a splitting to
to the map $\pi_{\rmY}^*: \hat \rmh^{*, \bullet}(\rmE\times_{\rmG}\rmY, \rmM) \ra \hat \rmh^{*, \bullet}(\rmE\times_{\rmG}(\rmY \times \rmX), \rmM)$,
provided the slices of the motivic spectrum $\rmM$ are module spectra over the motivic Eilenberg-Maclane spectrum $\rmH({\mathbb Z}/\ell^{\nu})$.
\item
If $\rmY$ is a geometrically connected smooth scheme and
$\rmE \times _{\rmG} \rmY$ is direct limit of an ind-scheme $\{\rmE_n \times_{\rmG} \rmY|n \ge 1\}$ with
each $\rmE_n \times_{\rmG} \rmY$ geometrically connected smooth scheme,  and $\ell$ is a fixed prime different from $char (\k)$,
$(id_{\rmY_{\bar {\k}}} \times \tau_{\rmX}(\rmf))^*(1)$ identifies with $\tr(\rmf_{\rmY})^*(1)$ under the isomorphism:
\[\rmH^0_{et}((\rmE\times_{\rmG}\rmY)_{\bar \k}, {\mathbb Z}/\ell^{\nu}) \cong \invlimn \rmH^0_{et}((\rmE_n\times_{\rmG}\rmY)_{\bar {\k}}, {\mathbb Z}/\ell^{\nu}) {\overset {\cong} \ra} \rmH^0_{et}(\rmY_{\bar {\k}}, {\mathbb Z}/\ell^{\nu})\]
where $(\rmE_n\times_{\rmG}\rmY)_{\bar {\k}}$ denotes the base change of $\rmE_n\times_{\rmG}\rmY$ to $Spec \, \bar \k$.
\item Therefore, under the above assumptions on $\rmY$ and $\rmE\times_{\rmG} \rmY$ as in (ii), and on the spectrum $\rmM$ as in (i),
the condition that $\tr(\rmf_{\rmY})^*(1) \in \rmH^0_{et}((\rmE\times_{\rmG}\rmY)_{\bar {\k}}, {\mathbb Z}/\ell^{\nu})$ is 
a unit proves that $\tr(\rmf_{\rmY})^*$ provides a splitting to
to the map $\pi_{\rmY}^*: \hat \rmh^{*, \bullet}(\rmE\times_{\rmG}\rmY, \rmM) \ra \hat \rmh^{*, \bullet}(\rmE\times_{\rmG}(\rmY \times \rmX), \rmM)$.
 \end{enumerate}
\end{proposition}
%\vskip .2cm \noindent
%{\bf Proof of Proposition ~\ref{splitting.et}}. 
\begin{proof} We already proved in Proposition ~\ref{transf.real} %and ~\ref{transf.real.2} 
that the pre-transfer is compatible with \'etale realization. In fact, one may take the 
commutative motivic ring spectrum $\cE$ to be the motivic Eilenberg-Maclane spectrum $\H({\mathbb Z}/\ell^{\nu})$. Then \cite[Theorem 7.3]{CJ23-T1}
shows $\cE \wedge X_+$ is dualizable in $\Spt(\k_{\rm mot}, {\cE})$ and Proposition ~\ref{transf.real} %and ~\ref{transf.real.2} 
shows that the pre-transfer is compatible with \'etale realizations. (Here are some more details
on the above arguments:
Let ${\mathbb H}(\rmR)$ denote the motivic spectrum representing motivic cohomology with respect to the commutative ring $\rmR$.
Now invoking  \cite[Proposition 8.3]{CJ23-T1}, we see that 
\be \begin{align}
 \label{trace.unit.slice.compl.ths}
(id _{\rmY_{+}}\wedge \tau_{\rmX}(\rmf))^*(1) &= (id _{\rmY_{+}} \wedge  {\mathbb H}(\rmR) \wedge \tau_{\rmX}(\rmf))^*(1)\\
                                                &= (id _{\rmY_{+}} \wedge \tau_{{\mathbb H}(\rmR) \wedge \rmX_+}(id \wedge \rmf))^*(1). \notag
\end{align} \ee
\vskip .1cm \noindent
Next we take $\rmR ={\mathbb Z}/\ell^{\nu}$, for $\ell$ prime to $char (\k)$.)
\vskip .1cm
These prove the first 
statement and the second statement is clear.
 The third statement then follows from the second in view of Theorem ~\ref{stable.splitting.2}(iii).
\end{proof}

\begin{remark}
%\begin{enumerate}[\rm(i)]
 %\item 
 If the base field is the field of complex numbers $\Cl$, one obtains similar results with respect to 
 Betti realization: the corresponding hypotheses will be that the schemes $\rmY$ and $\rmE\times_{\rmG}\rmY$ and the
 spaces $\rmY(\Cl)$ and $(\rmE\times_{\rmG}\rmY)(\Cl)$ are 
  connected,
  so that restriction maps $\rmH^{0, 0}(\rmE\times_{\rmG}\rmY, {\mathbb Z}) \ra \rmH^{0,0}(\rmY, {\mathbb Z})$ and
  $\rmH^{0}((\rmE\times_{\rmG}\rmY)(\Cl), {\mathbb Z}) \ra \rmH^{0}(\rmY(\Cl), {\mathbb Z})$ are isomorphisms, where the first (second)
  map is on motivic (singular) cohomology. Then, the 
  conclusion will be that the map $\pi_{\rmY}^*$ on motivic cohomology will
  be a split injection, provided $\tr(\rmf_{\rmY(\Cl)})^*(1)$ is a unit in $\rmH^{0}((\rmE\times_{\rmG}\rmY)(\Cl), {\mathbb Z})$.
  This is under the assumption that the transfer is compatible with Betti realization, which we have not proved in detail: see 
  \cite{Bain} for a proof in the case of the pre-transfer.
\end{remark}

\vskip .2cm
For the motivic spectrum representing Algebraic K-theory, the slice completed generalized motivic cohomology identifies 
with Algebraic K-theory. For a smooth ind-scheme $\Y=\{\rmY_m|m\}$, we let its algebraic
 K-theory spectrum be defined as ${\mathbf K}(\Y) = \holimm \{{\mathbf K}(\rmY_m)|m\}$. This provides the following corollary.
\begin{corollary} 
\label{Kth.splitting}
Let $\pi_{\rmY}:\rmE\times_{\rmG}(\rmY \times \rmX) \ra \rmE\times_{\rmG} \rmY$ and $\rmf$ denote maps as in Theorem ~\ref{stable.splitting.2}, with the 
group $\rmG$ any linear algebraic group that is special.
\begin{enumerate}[\rm(i)]
\item Then, 
$\pi_{\rmY}^*:{\mathbf K}(\rmE\times_{\rmG} \rmY) \ra {\mathbf K}(\rmE{\underset {\rmG}\times}(\rmY \times \rmX))$ is a split injection on homotopy groups, where ${\mathbf K}$ denotes the
motivic spectrum representing Algebraic K-theory, provided ${\Sigma^{\infty}_{\T}}\rmX_+$ is dualizable in $\SH(\k)$ and 
 $\tr(\rmf_{\rmY})^*(1)$ is 
a unit in $\rmH^{0, 0}(\rmE\times_{\rmG} \rmY, {\mathbb Z}) \cong \CH^0(\rmE\times_{\rmG} \rmY)$. In particular, this holds for smooth quasi-projective schemes
$\rmX$, $\rmY$ defined over the field $\k$ with $char (k) =0$, provided the above condition on $\tr(\rmf_{\rmY})^*(1)$ holds.
\vskip .1cm
\item Assume in addition to the above situation that the hypotheses in Proposition ~\ref{splitting.et}(ii) are satisfied.
  Then,
\[\pi_{\rmY}^*:{\mathbf K}(\rmE\times_{\rmG} \rmY)\wedge \rmM(\ell^{\nu}) \ra {\mathbf K}(\rmE\times_{\rmG}(\rmY \times \rmX)) \wedge \rmM(\ell ^{\nu})\]
is a split injection on homotopy groups, where $\rmM(\ell^{\nu})$ denotes the Moore spectrum defined
as the homotopy cofiber ${\Sigma^{\infty}_{\T}} {\overset {\ell^{\nu}} \ra} {\Sigma^{\infty}_{\T}}$, provided the following hold:
 ${\Sigma^{\infty}_{\T}}\rmX_+$ is dualizable in $\SH(\k, \mbS[\rmp^{\rm -1}])$, where $p$ is the characteristic exponent of $\k$,
 and $\tr(\rmf_{\rmY_{\bar \k}})^*(1)$ is a unit in $ \rmH^0_{et}((\rmE\times_{\rmG} \rmY)_{\bar \k}, Z/\ell^{\nu})$.
In particular, this holds for smooth quasi-projective schemes $\rmX$ defined over the field $\k$, with $char (\k)=p$, provided the above hypotheses hold.
\end{enumerate}
\end{corollary}
 \begin{proof}% {\bf Proofs of Corollary ~\ref{Kth.splitting}}. 
 The slice completed generalized motivic cohomology of any smooth scheme with respect to the motivic spectrum representing Algebraic K-theory,
 identifies with Algebraic K-theory itself. This proves the first statement in  Corollary ~\ref{Kth.splitting}.
 The second statement in Corollary ~\ref{Kth.splitting} now follows from the following observations.
 \vskip .2cm
 First we observe the weak-equivalence for any motivic spectrum $\cE$: $s_p(\cE) \wedge {\Sigma^{\infty}_{\T}}\rmM(\ell^{\nu}) \simeq s_p(\cE\wedge _{{\Sigma^{\infty}}_\T}\rmM(\ell^{\nu}))$, where
 $\rmM(\ell^{\nu})$ is defined as the homotopy cofiber of the map ${\Sigma^{\infty}_{\T}} {\overset {\ell^{\nu}} \ra} {\Sigma^{\infty}_{\T}}$, and where $s_p$ denotes the
 $p$-th slice. This follows from the identification of the slices, with the slices obtained from the coniveau tower as in
 \cite[Theorem 9.0.3]{Lev}.
 Let ${\mathbf K}$ denote the motivic spectrum representing algebraic K-theory.
 Next we recall (see \cite[section 11.3]{Lev}) that the slice $s_0({\mathbf K}) = \H({\mathbb Z})$ = the motivic Eilenberg-Maclane
 spectrum and that the $p$-th slice $s_p({\mathbf K}) = \H({\mathbb Z}(p)[2p])$, which is the corresponding shifted motivic
 Eilenberg-Maclane spectrum. Therefore, each $s_p({\mathbf K})$ has the structure of a module spectrum over the commutative ring spectrum
 $\H({\mathbb Z})$. In view of this, one may also observe that the natural map $s_p({\mathbf K})\wedge _{{\Sigma^{\infty}_{\T}}} \rmM(\ell^{\nu}) \ra 
 s_p({\mathbf K}) \wedge_{\H{\mathbb Z}} \H({\mathbb Z}/\ell^{\nu}) = \H({\mathbb Z}/\ell ^{\nu}(p)[2p])$ is a weak-equivalence, where $\H({\mathbb Z}/\ell^{\nu})$ denotes the 
 mod$-\ell^{\nu}$ motivic Eilenberg-Maclane spectrum. Therefore, the slices $s_p({\mathbf K}\wedge _{{\Sigma^{\infty}_{\T}}} \rmM(\ell^{\nu})) \simeq
 s_p({\mathbf K})\wedge _{{\Sigma^{\infty}_{\T}}} \rmM(\ell^{\nu})$ have the structure of weak-module spectra over the motivic
 Eilenberg-Maclane spectrum $\H({\mathbb Z}/\ell^{\nu})$. Therefore, the hypotheses of the statement in Theorem ~\ref{stable.splitting.2}(iii)
 are met.
 \vskip .1cm
 The additional assumptions then verify that the hypotheses of Proposition ~\ref{splitting.et}(iii) are also met
 thereby completing the proof of the second statement in Corollary ~\ref{Kth.splitting}. (One may also want to observe that
the spectrum ${\mathbf K}\wedge _{{\Sigma^{\infty}_{\T}}} \rmM(\ell^{\nu})$ has cohomological descent on the Nisnevich site of smooth schemes of finite type over $k$ so that
the generalized cohomology ${\rm h}(\rmX,{\mathbf K}\wedge _{{\Sigma^{\infty}_{\T}}} \rmM(\ell^{\nu})) \simeq {\mathbf K}(\rmX) \wedge _{{\Sigma^{\infty}_{\T}}} \rmM(\ell^{\nu})$ for any smooth scheme $\rmX$ of finite type over $k$.)
\end{proof}
 
\vskip .2cm %\noindent
%\vskip .1cm
The only {\it disadvantages} for this method seems to be that we need
to assume that the base $\rmB$ of the torsor is connected, the object $\rmY$ is a geometrically connected smooth scheme of finite type over $\k$, $\rmG$ is assumed to be special, 
%or , 
and also because this method applies only to slice-completed generalized motivic cohomology theories. However, as several
important examples of generalized motivic cohomology theories, such as Algebraic K-theory and Algebraic Cobordism
are slice-complete, there do not seem to be any serious disadvantages.
\normalfont
\vskip .3cm
\section{\bf Appendix: Further details on \'Etale realization}
As we pointed out already in Remark ~\ref{remark.et.real}, there is a discussion of \'etale realization in the unstable setting in \cite{Isak} and also \cite{Schm}. 
Though the discussions there have the expected good properties, these do not seem to extend to define an \'etale realization in the stable setting with good properties, partially because of
the following: the target of the realization functors discussed in \cite{Isak} and \cite{Schm} is the pro-category of inverse systems of
simplicial sets provided by applying the connected component functor degree-wise to the inverse systems of (rigid) \'etale hypercoverings. 
Unfortunately, the model structure on pro-simplicial sets (or for pro-objects in general model categories) 
that have been discussed in the literature, such as those in \cite{Isak2} are {\it not cofibrantly generated} as in fact pointed out in op. cit. As a result, obtaining a suitable 
stable model structure for the category of spectra in such pro-categories with good properties seems doubtful.
\vskip .1cm
We will circumvent these problems, and will sketch here a different process of \'etale realization that works stably without any of these difficulties stemming from having
to deal with pro-objects. This will also be a straightforward extension of the \'etale realization functors briefly discussed earlier in section ~\ref{et.real.1}.
%%%
\vskip .1cm
As is well-known, \'etale cohomology is not well-behaved unless one restricts to torsion coefficients, with torsion prime to the residue characteristics.
This shows therefore, the key role played by Bousfield-Kan type completion for spectra in \'etale setting. We will first discuss such completions
in the ${\mathbb A}^1$-setting.
\vskip .2cm

\subsection{$\Zl$-completions and spectra of $\Zl$-vector spaces}
\label{Zl-completions}
We will assume the framework of \cite[section 1]{CJ23-T1} in what follows. Throughout the following discussion, we will let ${\bZ}$ denote the ring of 
integers.
 \begin{enumerate}[\rm(i) ]
\item
For a pointed simplicial presheaf $\rmP \eps \Spc_*^{\rmG}(\k_{\rm mot})$, $\Zl(\rmP)$ will
denote the presheaf of simplicial $\Zl$-vector spaces defined in \cite[Chapter,
2.1]{B-K}. Observe that for each $n \ge 0$, $\Zl(\rmP)_n$ is the quotient of the free $\Zl$-module on $\rmP_n$ 
by $\Zl(*)$, where $*$ is the base point of $\rmP_n$. It follows readily that 
\be \begin{equation} 
\label{Zl.ident.1}
\Zl(\rmP \wedge \rmQ) \cong \Zl(\rmP)\otimes_{\Zl}\Zl(\rmQ).
\end{equation} \ee

\item 
Observe also that the presheaves of homotopy  groups of $\Zl(\rmP)$ identify with the
reduced homology presheaves of $\rmP$ with respect to the ring $\Zl$. Hence these
are all $\ell$-primary torsion. The functoriality of this construction shows that if $\rmP$ has
an action by the  group $\group$, then $\Zl(\rmP)$ inherits this action. 

\item
Next one extends the construction in the previous step to spectra: as we work largely with the category of spectra ${\widetilde {\Spt}}^{\rmG}(\k_{\rm mot})$ 
and ${\widetilde {\Spt}}^{\rmG}(\k_{et})$ we will only consider these categories.
Let $\X \in {\widetilde {\Spt}}^{\rmG}(\k_{\rm mot})$. Then first one
applies
the functor $(\quad ) \mapsto \Zl(\quad)$ to each $\X(\rmT_{\rmW})$, $\rmT_{\rmW} \eps \Sph^{\rmG}$. In view of ~\eqref{Zl.ident.1},  there exists a
natural map 
\[ \Zl(\Hom_{\Spc_*^{\rmG}(\k_{\rm mot})}(\rmT_{\rmU}, \rmT_{\rmU} \wedge \rmT_{\rmV})) {\underset {\Zl} \otimes} \Zl(\X(\rmT_{\rmW})) \ra \Zl(\Hom_{\Spc_*^{\rmG}(\k_{\rm mot})}(\rmT_{\rmU}, \rmT_{\rmU} \wedge \rmT_{\rmV}) \wedge \X(\rmT_{\rmW})),\]
where $\Hom_{(\Spc_*^{\rmG}(\k_{\rm mot}))}$ denotes the internal hom in the category ${\Spc_*^{\rmG}(\k)}$.
Therefore, one may compose the above maps with the obvious
map $\Zl(\Hom_{\Spc_*^{\rmG}(\k_{\rm mot})}(\rmT_{\rmU}, \rmT_{\rmU} \wedge \rmT_{\rmV})\wedge \X(\rmT_{\rmW})) \ra \Zl(\X(\rmT_{\rmV} \wedge \rmT_{\rmW}))$ to define an object in 
${\widetilde {\Spt}}^{\rmG}(\k_{\rm mot}, {\Zl}(\mbS^{\rmG}))$.

\item
A pairing $\rmM \wedge \rmN \ra \rmP$ in $\Spc_*^{\rmG}(\k_{\rm mot})$ induces a pairing
 $ \Zl(\rmM){\underset {\Zl} \otimes} \Zl(\rmN) 
\cong \Zl(\rmM \wedge \rmN) \ra \Zl(\rmP)$. Similarly 
a pairing $\X \wedge \Y \ra \Z$ in ${\widetilde {\Spt}}^{\rmG}(\k_{\rm mot})$ induces a similar 
pairing 
$ \Zl(\X) {\underset {\Zl} \otimes }\Zl(\Y) \ra \Zl(\Z)$. 
(To see this, one needs to recall the construction of the smash-product of spectra from \cite[Definition 4.2]{CJ23-T1} as a left-Kan extension.
The functor $\Zl$ commutes with this left-Kan extension.) 
This shows that if 
$\cE \in {\widetilde {\Spt}}^{\rmG}(\k_{\rm mot})$ is a ring spectrum so is $\Zl(\cE)$ and that if
 $\M \in {\widetilde {\Spt}}^{\rmG}(\k_{\rm mot})$ 
is a module spectrum over the ring spectrum $\cE$, $\Zl(\M)$ is a module spectrum over $\Zl(\cE)$.
\item
We will assume that ${\widetilde {\Spt}}^{\rmG}(\k_{\rm mot})$ is provided with the stable injective model structure, so that every object in ${\widetilde {\Spt}}^{\rmG}(\k_{\rm mot})$
is cofibrant.
If $\{\rmf: \rmA \ra \rmB\} =\oJ _{Sp}$ is a set of generating cofibrations (trivial cofibrations) for ${\widetilde {\Spt}}^{\rmG}(\k_{\rm mot})$, then,
 we will let $\{\Zl(\rmf): \Zl(\rmA) \ra \Zl(\rmB)|f \eps \oJ _{Sp}\}$ denote the set of generating  cofibrations 
(trivial cofibrations, \res) for
${\widetilde {\Spt}}^{\rmG}(\k_{\rm mot}, \Zl(\mbS^{\rmG}))$. This defines a model structure on ${\widetilde {\Spt}}^{\rmG}(\k_{\rm mot}, \Zl(\mbS^{\rmG}))$.
\end{enumerate} 
 The corresponding category of $\rmG$-equivariant spectra on the big \'etale site of $\Speck$ 
will be denoted \\
${\widetilde {\Spt}}^{\group}({\k_{et}}, \Zl(\mbS^{\rmG}_{\k,et}))$. These categories are also locally presentable and hence the model
categories are combinatorial.
One observes that the functor $\Zl(\quad )$
 induces a functor 
\[\Zl(\quad ): {\widetilde {\Spt}}^{\group}({\k_{\rm mot}}) ={\widetilde {\Spt}}^{\group}({\k_{\rm mot}}, \mbS^{\rmG}_{\k})\ra {\widetilde {\Spt}}^{\group}({\k_{\rm mot}}, \Zl(\mbS^{\rmG}_{\k})) \mbox{ and }\]
\[\Zl(\quad ): {\widetilde {\Spt}}^{\group}({\k_{et}}) ={\widetilde {\Spt}}^{\group}({\k_{et}}, \mbS^{\rmG}_{\k, et})\ra {\widetilde {\Spt}}^{\group}({\k_{et}}, \Zl(\mbS^{\rmG}_{\k,et})). \]
\vskip .3cm \noindent
Using the adjunction between the functors $\Zl(\quad) $ and $\rmU$, one may show readily that the functor $\Zl(\quad)$ is a
left Quillen functor and preserves weak-equivalences between cofibrant objects. However, in order to produce a triple,
using the above functors $\Zl(\quad ) $ and $\rmU$, one needs to compose the functor $\Zl(\quad ) $ with a fibrant replacement  functor.
This composition will be denoted
by ${\widetilde {\Zl}}(\quad ) $.

\vskip .2cm
\begin{proposition} 
 \label{funct.fib.repl}
 Let ${\widetilde {\Spt}}^{\group}({\k_{\rm mot}}, \Zl(\mbS^{\rmG}_{\k}))_{\it f} $ denote the fibrant objects in ${\widetilde {\Spt}}^{\group}({\k_{\rm mot}}, \Zl(\mbS^{\rmG}_{\k})) $.
\begin{enumerate}[\rm(i)]
 \item There exists a  functor $\X \ra {\widetilde {\Zl}(\X)}:
 {\widetilde {\Spt}}^{\group}({\k_{\rm mot}}, \mbS^{\rmG}_{\k}) \ra {\widetilde {\Spt}}^{\group}({\k_{\rm mot}}, \Zl(\mbS^{\rmG}_{\k}))_{\it f }$ left-adjoint to the forgetful functor
$\rmU:{\widetilde {\Spt}}^{ \group}(\k_{\rm mot}, \Zl(\mbS^{\rmG}_{\k}))_{\it f} \ra {\widetilde {\Spt}}^{\group}({\k_{\rm mot}}, \mbS^{\rmG}_{\k} )$. This functor preserves ${\mathbb A}^1$-fibrant spectra
and ${\mathbb A}^1$-stable weak-equivalences.
\item
The two adjoint functors $\rmU$ and ${\widetilde {\Zl}}(\quad)$ define a triple and hence a cosimplicial object with ${\widetilde {\Zl}}(\quad)$
in degree $0$. This cosimplicial object is group-like, i.e., each ${\widetilde {\Zl}}^n(\X)$ belongs to ${\widetilde {\Spt}}^{\rmG}(\k_{\rm mot}, \Zl(\mbS^{\rmG}_{\k}))$ and
all the cosimplicial structure maps are maps of $\Zl$-vector spaces, except for $d^0$. Therefore, it is fibrant in the Reedy-model structure on
cosimplicial objects in ${\widetilde {\Spt}}^{\group}({\k_{\rm mot}}, \Zl(\mbS^{\rmG}_{\k}))$.
\item The corresponding results also hold for the functor ${\widetilde {\Zl}}(\quad ): {\widetilde {\Spt}}^{\group}({\k_{et}}, \mbS^{\rmG}_{\k,et})\ra {\widetilde {\Spt}}^{\group}({\k_{et}}, \Zl(\mbS^{\rmG}_{\k,et}))$.
\end{enumerate}
\end{proposition}
\begin{proof} We will only sketch the outlines of a proof as the above discussion should be known to the experts. One defines the 
 functor $\X \ra {\widetilde {\Zl}(\X)}$ by taking a functorial fibrant replacement of the functor $ \X \ra {\Zl}(\X)$. Then given a map
 $\rmA \ra \rmU(\tilde \rmE)$, with $\tilde \rmE$ in ${\widetilde {\Spt}}^{\group}({\k_{\rm mot}}, \Zl(\mbS^{\rmG}_{\k}))_{\it f} $, this corresponds
 by adjunction to a map ${\Zl}(\rmA) \ra \tilde \rmE$. Then one shows that this map admits a lifting to a map ${\widetilde {\Zl}}(\rmA) \ra \tilde \rmE$
 using a standard construction of the functorial fibrant replacement using a small object argument: see, for example, \cite[Theorem 2.1.14]{Hov99}. This proves (i) and 
 the remaining statements follow.
\end{proof}

\vskip .2cm
\begin{definition}
\label{l.completion}
$ {\widetilde {\Zl}} _{\infty} (\X) = \holimD \{{\widetilde {\Zl}}^n(\X)|n\}$.
\end{definition}
\begin{proposition}
\label{props.l.completion}
The functor $\X \mapsto {\widetilde {\Zl}} _{\infty} (\X)$, ${\widetilde {\Spt}}^{\rmG}(\k_{\rm mot}) \ra {\widetilde {\Spt}}^{\rmG}(\k_{\rm mot})$ has the following properties:
\begin{enumerate}[\rm(i)]
\item{For $\X \in {\widetilde {\Spt}}^{\rmG}(\k_{\rm mot})$, $\widetilde {Z/l_{\infty}(\X)}$ is $\Zl$-complete, 
i.e., (i) it is fibrant in
 ${\widetilde {\Spt}}^{\group}(\k_{\rm mot})$ and for every map $\phi: \rmA \ra \rmB$ in ${\widetilde {\Spt}}^{\group}({\k_{\rm mot}})$ between cofibrant 
objects which induces an isomorphism
 \[\pi_*({\widetilde {\Zl}}(\rmA)) {\overset {\cong} {\ra}} \pi_*({\widetilde {\Zl}}(\rmB))\] of the reduced homology
presheaves,  the induced map
 $\phi^*:Map(\rmB, \widetilde {\Zl_{\infty}(\X)}) \ra Map(\rmA, \widetilde
{\Zl_{\infty}(\X)})$ is a weak-equivalence of simplicial sets, where $\Map$ denotes the simplicial mapping space functor. In particular, this applies to the case where $\rmB= {\widetilde {\Zl}}_{\infty}(\rmA)$ and
$\phi: \rmA \ra \rmB$ is the obvious Bousfield-Kan completion map.}
\item{If $\cE \eps {\widetilde {\Spt}}^{\group}(\k_{\rm mot})$ is a ring spectrum so are each $\widetilde {\Zl_n(\cE)}$
and ${\Zl_{\infty}(\cE)}$. If $\cE \in {\widetilde {\Spt}}^{\group}(\k_{\rm mot})$ is a ring spectrum 
and $\M \in  {\widetilde {\Spt}}^{\group}(\k_{\rm mot})$ is a module spectrum  over $\cE$, then
each ${\widetilde {\Zl_n(\M)}}$ is a ${\widetilde {\Zl_n(\cE)}}$-module spectrum and ${\widetilde {\Zl}}_{\infty}(\M)$ 
is a ${\widetilde {\Zl}}_{\infty}(\cE)$-module spectrum.}
\item{For each $\cE \eps {\widetilde {\Spt}}^{\group}(\k_{\rm mot})$, each 
 $\widetilde {\Zl_n(\cE)}$ is ${\mathbb A}^1$-local and belongs to ${\widetilde {\Spt}}^{\group}(\k_{\rm mot})$.}
\item The corresponding results also hold for the functor ${\widetilde {\Zl}} _{\infty} : {\widetilde {\Spt}}^{\group}({\k_{et}})\ra {\widetilde {\Spt}}^{\group}({\k_{et}})$.
\end{enumerate}

\end{proposition}
\begin{proof} Here we will again provide only a sketch of (i), as the remaining statements should be clear. A key observation is that 
 the map $\phi$ induces a weak-equivalence ${\Zl}(\phi): {\Zl}(\rmA) \ra {\Zl}(\rmB)$, so that the induced map $\phi^*:Map({\Zl}(\rmB), \widetilde {\Zl(\X)}) \ra Map({\Zl}(\rmA), \widetilde
{\Zl(\X)})$ is a weak-equivalence. Now (i) follows readily from this observation.
\end{proof}

\subsection{The refined \'etale realization}
\vskip .3cm
\begin{definition} (The refined \'etale realization functor)
\label{et.real.2}
 We define the refined \'etale realization functor as the composition $Et= {\widetilde {\Zl}}_{\infty} \circ \epsilon^*: {\widetilde {\Spt}}^{\rmG}(\k_{\rm mot}) \ra {\widetilde {\Spt}}^{\rmG}(\k_{et}, {\widetilde {\Zl}}_{\infty} (\mbS^{\rmG}_{\k,et}))$.
\end{definition} \index{$Et$} \index{refined \'etale realization}
\begin{proposition}
\label{ref.etal.real}
\begin{enumerate}[\rm(i)]
\item The above \'etale realization functor is weakly-monoidal, i.e., for each pairing $\X \wedge \Y \ra \Z$ of spectra in $\Spt(\k_{\rm mot})$,
there is a natural induced pairing $Et(\X) \wedge Et(\Y) \ra Et(\Z)$.
\item Assume the base field $k$ is separably closed, with characteristic $p$. Then for any prime $\ell \ne p$, the two categories of spectra
 ${\widetilde {\Spt}}(\k_{et}, {\widetilde {\Zl}}_{\infty} (\mbS_{\T, et}))$ and ${\widetilde {\Spt}}(\k_{et}, {\widetilde {\Zl}}_{\infty} (\mbS_{\rmS^2, et}))$ are equivalent,
 where $\mbS_{\T, et}$ ($\mbS_{\rmS^2, et}$) denotes the sphere spectrum defined as the $\T$-suspension spectrum (the $\rmS^2$-suspension spectrum, \res) of
 $\rmS^0$.
 \end{enumerate}
\end{proposition}
\begin{proof} (i) is clear in view of the properties of the completion functor discussed in Proposition ~\ref{props.l.completion}.
 To see (ii), one uses the following well-known argument. Let $\rmW(\k)$ denote the ring of Witt vectors of $\k$. This is a Hensel ring with
 closed point corresponding to $\k$  and the generic point in characteristic $0$. Therefore, one may imbed $\rmW(\k)$ into the field of complex numbers $\Cl$.
 Now the projective space ${\mathbb P}^1$ over $\k$ lifts to the projective space ${\mathbb P}^1$ over $\Cl$. On taking the $\Zl$-completion, one sees that
 ${\widetilde {\Zl}}_{\infty}({\mathbb P}^1) \simeq  {\widetilde {\Zl}}_{\infty}({\rm S}^2)$.
\end{proof}

\vskip .2cm
We conclude with the following result.
\begin{proposition} Assume the base scheme is a separably closed field $\k$  and
$char(\k) =p$. Let $E \eps \Spt(\k_{et}) $ be a constant sheaf of spectra so
that all the (sheaves of) homotopy groups
$\pi_n(E) $ are $\ell$-primary torsion, for some $\ell \ne p$. Then $E$ is ${\mathbb
A}^1$-local in $\Spt(\k_{et})$, i.e., the projection $\rmP \wedge {\mathbb
A}^1_{+} \ra \rmP$ induces a weak-equivalence: $Map (\rmP, E) \simeq Map (\rmP
\wedge {\mathbb A}^1_{ +}, E)$, $\rmP \eps \Spt(\k_{et})$.
\end{proposition}
\begin{proof} First let $\rmP$ denote the suspension spectrum associated to some smooth
scheme $\rmX \in \Sm_{\k}$. Then ${\rm Map}(\rmP, E)$ (${\rm Map} (\rmP \wedge {\mathbb A}^1_{ +},
E)$) identifies with the spectrum defining the generalized \'etale cohomology
of $\rmX$ (of $\rmX \times {\mathbb A}^1$, \res) with respect to the spectrum $E$.
There exist Atiyah-Hirzebruch spectral sequences that converge to these
generalized \'etale cohomology groups with the $E_2^{s, t}$-terms being $H^s_{et}(\rmX,
\pi_{-t}(E))$
and $H^s_{et}(\rmX \times {\mathbb A}^1_{}, \pi_{-t}(E))$, \res. Since the
sheaves of homotopy groups $\pi_{-t}(E)$ are all $\ell$-torsion with $\ell \ne p$, and
$\k$ is assumed to be a separably closed field, $\rmX$ and $\rmX \times {\mathbb A}^1_{}$
have finite $\ell$-cohomological dimension. Therefore these spectral sequences
converge strongly and the conclusion of the proposition holds in this
case. For a general simplicial presheaf $\rmP$, one may find a simplicial
resolution 
where each term is a disjoint union of schemes as above (indexed by a small
set). Therefore, the conclusion of the proposition holds also for suspension
spectra
 of all simplicial presheaves and therefore for all spectra $\rmP$. \end{proof}
%\vskip .3cm
\vfill \eject
 %%%%%%%%%%%%%%%%%%%%%%%%%%%%%%%%%%%%%%%%%%%%%%%%%%%%%%%%%%%%%%%%%%%%%
%%%%%%%%%%%%%%%%%%%%%%%%%%%%%%%%%%%%%%%%%%%%%%%%%%%%%%%%%%%%%%%%%%%%%
%% REFERENCES
%%%%%%%%%%%%%%%%%%%%%%%%%%%%%%%%%%%%%%%%%%%%%%%%%%%%%%%%%%%%%%%%%%%%%

%%%%%%%%%%%%%%%%%%%%%%%%%%%%%%%%%%%%%%%%%%%%%%%%%%%%%%%%%%%%%%%%%%%%%
%\vskip -.4cm

\end{document}